\documentclass[reqno]{amsart}
\usepackage{amscd}
\usepackage{amssymb}
\usepackage[margin=1in,includeheadfoot]{geometry}
\usepackage{cite}
\usepackage [latin1]{inputenc}
\usepackage{tabularx,booktabs,tikz}
\usepackage{caption}
\usepackage{amsmath}
\usepackage{amsfonts}

\usepackage{amscd}
\usepackage{amsthm}
\usepackage{amssymb} \usepackage{latexsym}
\usepackage{eufrak}
\usepackage{euscript}
\usepackage{epsfig}
\usepackage{graphics}
\usepackage{array}
\usepackage{enumerate}
\usepackage{dsfont}
\usepackage{color}
\usepackage{wasysym}
\usepackage{hyperref}
\usepackage{pdfsync}

\def\beq{\begin{equation}}
\def\eeq{\end{equation}}
\def\ba{\begin{array}}
\def\ea{\end{array}}

\newtheorem{thm}{Theorem}[section]
\newtheorem{lm}[thm]{Lemma}
\newtheorem{prop}[thm]{Proposition}
\newtheorem{crl}[thm]{Corollary}

\theoremstyle{definition}
\newtheorem{rem}[thm]{Remark}

\theoremstyle{remark}

\begin{document}
\pagestyle{plain}
\title{Sign-changing solutions to discrete nonlinear logarithmic Kirchhoff equations}

\author{Lidan Wang}
\email{wanglidan@ujs.edu.cn}
\address{Lidan Wang: School of Mathematical Sciences, Jiangsu University, Zhenjiang 212013, People's Republic of China}

\begin{abstract}
In this paper, we study the discrete logarithmic Kirchhoff equation
$$
-\left(a+b \int_{\mathbb{Z}^3}|\nabla u|^{2} d \mu\right) \Delta u+(\lambda h(x)+1) u=|u|^{p-2}u \log u^{2}, \quad x\in \mathbb{Z}^3,
$$
where $a,b>0, p>6$ and $\lambda$ is a positive parameter. Under suitable assumptions on $h(x)$, we prove the existence and asymptotic behavior of least energy sign-changing solutions for the equation by the method of Nehari manifold.
\end{abstract}

\maketitle


\
\


\section{Introduction}
The Kirchhoff-type equation
\begin{equation}\label{z}
    -\left(a+b \int_{\mathbb{R}^3}|\nabla u|^{2} d \mu\right) \Delta u+\omega(x) u=g(x,u),\quad u\in H^1(\mathbb{R}^3),
\end{equation}
where $a,\,b>0$, has drawn lots of interest in recent years due to the appearance of $(\int_{\mathbb{R}^3}|\nabla u|^{2}\,d \mu)\Delta u$.  See for examples \cite{CT,G,HQ,HZ,TS,W,WT} for the general nonlinearity $g(x,u)$. For the equations with nonlinearity of power type $g(x,u)=|u|^{p-2}u$, we refer the readers to \cite{LY,LL,SZ}. Moreover, for the equations with convolution nonlinearity $g(x,u)=\left(I_\alpha \ast F(u)\right)f(u)$, where $I_\alpha$ is the Riesz potential, we refer the readers to  \cite{CZ,H,LS,LD,ZZ1}. In particular, for $\omega(x)=(\lambda h(x)+1)$ and $f(u)=|u|^{p-2}u$ with $p\in(2,3+\alpha)$, L\"{u} \cite{L} demonstrated the existence and asymptotic behavior 
of ground state solutions by the
Nehari manifold and the concentration compactness principle.
For more related works about the Choquard-type nonlinearity, we refer the readers to \cite{GT,HT,LP,YG}.

In many physical applications, the logarithmic nonlinearity $g(x,u)=|u|^{p-2}u\log u^2$ in the equation (\ref{z}) appears naturally.
In particular, for $a=1, b=0$ and $p = 2$, the equation (\ref{z}) turns into the well
known logarithmic Schr\"{o}dinger equation, which has been studied extensively in the
literature, see for examples \cite{AJ1,AJ,FT,HL,S,ZZ2}. For $a,\,b>0$, Wen, Tang and Chen \cite{WTC} considered the logarithmic Kirchhoff equation with $p\in(4,6)$ in a bounded domain of $\mathbb{R}^3$ and established the existence of ground state solutions and
ground state sign-changing solutions with precisely two nodal domains. Li, Wang and Zhang \cite{LWZ} generalized the results of \cite{WTC}  to a class of $p$-Laplacian Kirchhoff-type
equations with logarithmic nonlinearity. Hu and Gao \cite{HG} studied the logarithmic Kirchhoff equation with $p\in(4,6)$ in $\mathbb{R}^3$ and proved the existence of positive solutions and sign-changing solutions by the constraint variational method and Brouwer degree theory under suitable assumptions on the potential $\omega(x)$. Later, Gao et al. \cite{GJ} considered a more complicated logarithmic Kirchhoff equation in $\mathbb{R}^3$ and obtained the existence of trivial solutions for large $b > 0$ and two positive solutions for small $b>0$.

Nowadays, many researchers  turn to study  differential equations on graphs, especially for the nonlinear elliptic equations. See for examples \cite{GLY,HSZ,HLW,HX,W2,ZZ} for the  discrete nonlinear  Schr\"{o}diner equations. For the  discrete nonlinear Choquard equations, we refer the readers to \cite{LW,LZ1,LZ,W1}.  For the equations with logarithmic nonlinearity on graphs, one can see \cite{CWY,CR,HJ,OZ,SY}. It is worth noting that the authors in \cite{CR}  established the existence and asymptotic behavior of least energy sign-changing solutions for a class of logarithmic Schr\"{o}dinger equations. Recently, L\"{u} \cite{L1} proved the existence of ground state solutions for a class of Kirchhoff equations on lattice graphs $\mathbb{Z}^3$. Wang \cite{W3} demonstrated the existence of nontrivial solutions and ground
state solutions for the general Kirchhoff-Choquard equations on lattice graphs $\mathbb{Z}^3$. To the best of our knowledge, there is no existence results for the logarithmic Kirchhoff equations on infinite graphs. Motivated by the works mentioned above, in this paper, we would like to study a class of Kirchhoff-type equations with logarithmic nonlinearity on lattice graphs $\mathbb{Z}^3$ and discuss the existence and asymptotic behavior of least energy sign-changing solutions.

Let us first give some notations. Let $\Omega$ be a subset of $\mathbb{Z}^3$, we denote by $C(\Omega)$ the set of functions on $\Omega$ and $C_{c}(\Omega)$ the set of all functions on $\Omega$ with finite support, where the support of $u\in C(\Omega)$ is
defined as $\text{supp}(u):=\{x \in \Omega : u(x)\neq 0\}$. Moreover, we denote by the $\ell^p(\Omega)$ the space of $\ell^p$-summable functions on $\Omega$. For convenience, for any $u\in C(\Omega)$, we always write
$
\int_{\Omega}u\,d\mu=\sum\limits_{x\in \Omega}u(x),$ where $\mu$ is the counting measure in $\Omega$.

In this paper, we consider the following logarithmic kirchhoff equation
\begin{equation}\label{0.2}
-\left(a+b \int_{\mathbb{Z}^3}|\nabla u|^{2} d \mu\right) \Delta u+(\lambda h(x)+1)u=|u|^{p-2}u \log u^{2}, \quad x\in\mathbb{Z}^3,
\end{equation}
where $a, b>0, p>6$ and $\lambda$ is a positive parameter. Here $\Delta u(x)=\sum\limits_{y \sim x}(u(y)-u(x))$ and $|\nabla u(x)|=\left(\frac{1}{2} \sum\limits_{y \sim x}(u(y)-u(x))^{2}\right)^{\frac{1}{2}}.$ 
We always assume that the potential $h$ satisfies 
\begin{itemize}
    \item[$(h_1)$] $h(x) \geq 0$ and the potential well $\Omega=\{x \in \mathbb{Z}^3: h(x)=0\}$ is a nonempty, connected and bounded domain in $\mathbb{Z}^3$;
    \item[$(h_2)$] there exists $M>0$ such that $\mu(\{x \in \mathbb{Z}^3: h(x)<M\})<\infty$.
\end{itemize}

Note that for any $q>p$, we have
$$\lim\limits_{t\rightarrow 0}\frac{t^{p-1}\log t^2}{t}=0, \qquad\text{and}\qquad \lim\limits_{t\rightarrow \infty}\frac{t^{p-1}\log t^2}{t^{q-1}}=0.$$
Then for any $\varepsilon>0$ , there exists $C_\varepsilon>0$ such that
\begin{equation}\label{1.2}
|t|^{p-1}|\log t^2|\leq \varepsilon |t|+C_\varepsilon |t|^{q-1}, \quad t\neq 0.
\end{equation}

Let $H^1(\mathbb{Z}^3)$ be the completion of $C_{c}(\mathbb{Z}^3)$ under the norm
$$
\|u\|_{H^1}=\left(\int_{\mathbb{Z}^3}\left(|\nabla u|^{2}+u^{2}\right) d \mu\right)^{\frac{1}{2}}.
$$
The space $H^1(\mathbb{Z}^3)$ is a Hilbert space with the inner product
$$
\langle u, v\rangle_{H^1}=\int_{\mathbb{Z}^3}(\nabla u \nabla v+u v) d \mu,
$$
where $\nabla u\nabla v=\frac{1}{2} \sum\limits_{y \sim x}(u(y)-u(x))(v(y)-v(x))$ is the gradient form.

For any $\lambda>0$, we introduce a space
$
\mathcal{H}_{\lambda}=\left\{u \in H^1(\mathbb{Z}^3): \int_{\mathbb{Z}^3} \lambda h(x) u^{2}\,d \mu<\infty\right\}
$
equipped with the norm
$$
\|u\|_{\mathcal{H}_{\lambda}}^{2}=\int_{\mathbb{Z}^3}\left(a|\nabla u|^{2}+h_\lambda(x) u^{2}\right)\,d \mu,
$$
where $h_\lambda(x)=(\lambda h(x)+1)$ and $a>0$.
The space $\mathcal{H}_{\lambda}$ is a Hilbert space with the inner product
$$
\langle u, v\rangle_{\mathcal{H}_{\lambda}}=\int_{\mathbb{Z}^3}(a\nabla u \nabla v+h_\lambda(x) u v)\,d \mu.
$$

The energy functional $J_{\lambda}: \mathcal{H}_\lambda\rightarrow \mathbb{R}$ associated to the equation (\ref{0.2}) is given by
\begin{equation*}
J_{\lambda}(u)=\frac{1}{2} \int_{\mathbb{Z}^3}\left(a|\nabla u|^{2}+h_\lambda(x) u^{2}\right) d \mu+\frac{b}{4}\left(\int_{\mathbb{Z}^3}|\nabla u|^{2} d \mu\right)^{2}+\frac{2}{p^2}\int_{\mathbb{Z}^3}|u|^p\,d\mu-\frac{1}{p} \int_{\mathbb{Z}^3} |u|^{p} \log u^{2} d \mu. 
\end{equation*}
By (\ref{1.2}), one gets easily that $J_\lambda\in C^1(\mathcal{H}_\lambda,\mathbb{R})$ and, for $\phi\in \mathcal{H}_\lambda$,
\begin{eqnarray*}
(J'_\lambda(u),\phi)=\int_{\mathbb{Z}^3}\left(a \nabla u \nabla \phi+h_\lambda(x) u \phi\right)\,d \mu+b \int_{\mathbb{Z}^3}|\nabla u|^{2}\,d \mu \int_{\mathbb{Z}^3} \nabla u \nabla \phi\,d \mu-\int_{\mathbb{Z}^3}|u|^{p-2}u\phi\log u^2\,d\mu.
\end{eqnarray*}
We say that $u\in\mathcal{H}_\lambda$ is a weak solution of the equation (\ref{0.2}), if $u$ is a critical point of the energy functional $J_\lambda$, i.e. $J'_\lambda(u)=0$.  A sign-changing solution of the equation (\ref{0.2}) means that $u$ is a weak solution of the equation (\ref{0.2}) such that $u^{\pm}\not\equiv 0$, where $u^{+}=\max \{u, 0\}$ and $u^{-}=\min \{u, 0\}$.

Now we define the Nehari manifold and sign-changing Nehari set respectively by
$$
\begin{aligned}
& \mathcal{N}_{\lambda}=\left\{u \in \mathcal{H}_{\lambda} \backslash\{0\}: (J_{\lambda}^{\prime}(u),u)=0\right\},\qquad \mathcal{M}_{\lambda}=\left\{u \in \mathcal{H}_{\lambda}: u^{ \pm} \neq 0 \text { and } (J_{\lambda}^{\prime}(u), u^{+})=(J_{\lambda}^{\prime}(u),u^{-})=0\right\}.
\end{aligned}
$$
 Clearly, $\mathcal{N}_{\lambda}$ contains all the nontrivial solutions of the equation (\ref{0.2}) and the set $\mathcal{M}_{\lambda}$ contains all the sign-changing solutions of the equation (\ref{0.2}). Let
$$
c_{\lambda}=\inf _{u \in \mathcal{N}_{\lambda}} J_{\lambda}(u), \quad m_{\lambda}=\inf _{u \in \mathcal{M}_{\lambda}} J_{\lambda}(u).
$$

Now we state our first result.

\begin{thm}\label{t1}
Let $p>6$ and $(h_1)$-$(h_2)$ hold. Then there exists a constant $\lambda_{0}>0$ such that, for all $\lambda \geq \lambda_{0}$, the equation (\ref{0.2}) has a least energy sign-changing solution $u_{\lambda} \in \mathcal{H}_{\lambda}$ satisfying $J_{\lambda}\left(u_{\lambda}\right)=m_{\lambda}$. Moreover, $m_{\lambda}>2 c_{\lambda}$.
\end{thm}

In order to study the asymptotic behavior of $u_\lambda$ as $\lambda\rightarrow\infty$, we introduce the limiting equation defined on the finite potential well $\Omega$
\begin{equation}\label{1.8}
 \begin{cases}-\left(a+b \int_{\Omega}|\nabla u|^{2} d \mu\right) \Delta u+u=|u|^{p-2}u \log u^{2}, & x\in\Omega,  \\ u(x)=0, & x\in \partial \Omega.
 \end{cases}
\end{equation}

Different from the continuous case, the Laplacian and the gradient form of $u$ on a bounded domain $\Omega$ need additional information on the vertex boundary of $\Omega$, which is defined by $$\partial\Omega=\{y\in\Omega^c :\exists~x\in\Omega~ \text{such~that}~y\sim x \},$$
where $\Omega^c=\mathbb{Z}^3\backslash\Omega.$

Let $\mathcal{H}^1_0(\Omega)$ be the completion of $C_c(\Omega)$ with respect to the norm
$$
\|u\|_{\mathcal{H}_{0}^{1}(\Omega)}=\left(\int_{\Omega \cup \partial \Omega}a|\nabla u|^{2} d \mu+\int_{\Omega} |u|^2\,d\mu\right)^{\frac{1}{2}}.
$$
The space  $\mathcal{H}_{0}^{1}(\Omega)$ is a Hilbert space with the inner product
$$
\langle u, v\rangle_{\mathcal{H}_{0}^1}=\int_{\Omega \cup \partial \Omega}a\nabla u \nabla v\,d\mu+\int_{\Omega}u v\,d \mu.
$$

The energy functional $J_{\Omega}: \mathcal{H}_{0}^{1}(\Omega) \rightarrow \mathbb{R}$ associated to the equation (\ref{1.8}) is given by
$$
J_{\Omega}(u)=\frac{1}{2}\int_{\Omega \cup \partial \Omega}a|\nabla u|^{2} d \mu+\frac{1}{2}\int_{\Omega}u^{2} d \mu+\frac{b}{4}\left(\int_{\Omega \cup \partial \Omega}|\nabla u|^{2} d \mu\right)^{2}+\frac{2}{p^2}\int_{\Omega}|u|^p\,d\mu-\frac{1}{p} \int_{\Omega} |u|^{p} \log u^{2} d \mu.
$$
Similarly, one gets that $J_\Omega\in C^1(\mathcal{H}^1_0,\mathbb{R})$ and, for $\phi\in \mathcal{H}^1_0(\Omega),$
$$(J'_\Omega(u),\phi)=\int_{\Omega \cup \partial \Omega}a \nabla u \nabla \phi\,d \mu+\int_{\Omega}u\phi\,d\mu+b \int_{\Omega \cup \partial \Omega}|\nabla u|^{2}\,d \mu \int_{\Omega \cup \partial \Omega} \nabla u \nabla \phi\,d \mu-\int_{\Omega}|u|^{p-2}u\phi\log u^2\,d\mu.$$
We say that $u\in\mathcal{H}^1_0(\Omega)$ is a weak solution of the (\ref{1.8}), if $u$ is a critical point of the energy functional $J_\Omega$, i.e. $J'_\Omega (u)=0$. A sign-changing solution of the equation (\ref{1.8}) means that $u$ is a weak solution of the equation (\ref{1.8}) such that $u^{\pm}\not\equiv 0$.

Let
$$
c_{\Omega}=\inf _{u \in \mathcal{N}_{\Omega}} J_{\lambda}(u), \quad m_{\Omega}=\inf _{u \in \mathcal{M}_{\Omega}} J_{\Omega}(u),
$$
where
$$
\begin{aligned}
& \mathcal{N}_{\Omega}=\left\{u \in \mathcal{H}_{0}^{1}(\Omega) \backslash\{0\}: (J_{\Omega}^{\prime}(u),u)=0\right\},\\& \mathcal{M}_{\Omega}=\left\{u \in \mathcal{H}_{0}^{1}(\Omega): u^{ \pm} \neq 0 \text { and } (J_{\Omega}^{\prime}(u),u^{+})=(J_{\Omega}^{\prime}(u),u^{-})=0\right\}.
\end{aligned}
$$

Our rest results are as follows.
\begin{thm}\label{t2}
  Let $p>6$ and $(h_1)$-$(h_2)$ hold. Then the equation (\ref{1.8}) has a least energy sign-changing solution $u_{0} \in \mathcal{H}_{0}^{1}(\Omega)$ satisfying $J_{\Omega}\left(u_{\Omega}\right)=m_{\Omega}$. Moreover, $m_{\Omega}>2 c_{\Omega}$.   
\end{thm}

\begin{thm}\label{t3}
Let $p>6$ and $(h_1)$-$(h_2)$ hold. Then for any sequence $\lambda_{k} \rightarrow\infty$, up to a subsequence, the corresponding least energy sign-changing solution $u_{\lambda_{k}}$ of the equation (\ref{0.2}) converges in $H^1(V)$ to a least energy sign-changing solution of the equation (\ref{1.8}).    
\end{thm}

This paper is organized as follows. In Section 2, we present some preliminary results on graphs. In Section 3, we prove the existence of least energy sign-changing solutions to the equation (\ref{0.2}) and the equation (\ref{1.8})(Theorem \ref{t1} and Theorem \ref{t2}). In Section 4, we verify the asymptotic behavior of least energy sign-changing solutions to the equation (\ref{0.2})(Theorem \ref{t3}).

\section{Preliminaries}
In this section, we introduce the basic settings on graphs and give some preliminary results. 

Let $G=(V, E)$ be a connected, locally finite graph, where $V$ denotes the vertex set and $E$ denotes the edge set. We call vertices $x$ and $y$ neighbors, denoted by $x \sim y$, if there exists an edge connecting them, i.e. $(x, y) \in E$. For any $x,y\in V$, the distance $d(x,y)$ is defined as the minimum number of edges connecting $x$ and $y$, namely
$$d(x,y)=\inf\{k:x=x_0\sim\cdots\sim x_k=y\}.$$
Let $B_{r}(a)=\{x\in V: d(x,a)\leq r\}$ be the closed ball of radius $r$ centered at $a\in V$. For brevity, we write $B_{r}:=B_{r}(0)$.

In this paper, we consider, the natural discrete model of the Euclidean space, the integer lattice graph.  The $3$-dimensional integer lattice graph, denoted by $\mathbb{Z}^3$, consists of the set of vertices $V=\mathbb{Z}^3$ and the set of edges $E=\{(x,y): x,\,y\in~V,\,\underset {{i=1}}{\overset{3}{\sum}}|x_{i}-y_{i}|=1\}.$
In the sequel, we denote $|x-y|:=d(x,y)$ on the lattice graph $~V$.

For $u,v \in C(V)$, we define the Laplacian of $u$ as
$$
\Delta u(x)=\sum\limits_{y \sim x}(u(y)-u(x)),
$$
 and the gradient form $\Gamma$ as
$$
\Gamma(u, v)(x)=\frac{1}{2} \sum\limits_{y \sim x}(u(y)-u(x))(v(y)-v(x))=:\nabla u\nabla v.
$$
We write $\Gamma(u)=\Gamma(u, u)$ and denote the length of the gradient as
$$
|\nabla u|(x)=\sqrt{\Gamma(u)(x)}=\left(\frac{1}{2} \sum\limits_{y \sim x}(u(y)-u(x))^{2}\right)^{\frac{1}{2}}.
$$

We denote by $C(\Omega)$ the set of functions on $\Omega\subset V$. The $\ell^p(\Omega)$ space is defined as
$$\ell^p(\Omega)=\{u\in C(\Omega):\|u\|_{\ell^p(\Omega)}<\infty\},$$
where
 $$
\|u\|_{\ell^p(\Omega)}= \begin{cases}\left(\sum\limits_{x \in \Omega}|u(x)|^{p}\right)^{\frac{1}{p}}, &  1 \leq p<\infty, \\ \sup\limits_{x \in \Omega}|u(x)|, & p=\infty.\end{cases}
$$
We shall write $\|u\|_p$ instead of $\|u\|_{\ell^p(V)}$ if $\Omega=V$.

\begin{prop}\label{o} Let $s,t>0$. Then for any $u\in\mathcal{H}_\lambda$, we have
  \begin{itemize}
      \item[(i)] $$
      \int_{V} \left|\nabla(su^{+}+tu^{-})\right|^2 d \mu=\int_{V} \left|\nabla (su^{+})\right|^2 d \mu+\int_{V}\left|\nabla (tu^{-})\right|^2 d \mu-stK_{V}(u),
      $$
 where $K_V(u)=\sum\limits_{x \in V} \sum\limits_{y \sim x}\left[u^{+}(x) u^{-}(y)+u^{-}(x) u^{+}(y)\right]\leq 0.$    
 \item[(ii)] $$
 \int_{V} \nabla\left(su^{+}+tu^{-}\right)\nabla (su^+)\, d \mu=\int_{V} \left|\nabla(s u^{+})\right|^2 d \mu-\frac{st}{2} K_{V}(u).
 $$
 \item[(iii)] $$
 \int_{V} \nabla\left(su^{+}+tu^{-}\right)\nabla (tu^-)\, d \mu=\int_{V} \left|\nabla (tu^{-})\right|^2 d \mu-\frac{st}{2} K_{V}(u).
$$
  \end{itemize}  
\end{prop}

\begin{proof}
(i) A direct calculation yields that
$$
\begin{aligned}
& \int_{V} \left|\nabla(su^{+}+tu^{-})\right|^2 d \mu \\
= & \frac{1}{2} \sum_{x \in V} \sum_{y \sim x}\left[\left(su^{+}+tu^{-}\right)(y)-\left(su^{+}+tu^{-}\right)(x)\right]^{2}\\
= & \frac{1}{2} \sum_{x \in V} \sum_{y \sim x}\left[\left(su^{+}(y)-su^{+}(x)\right)^{2}+\left(tu^{-}(y)-tu^{-}(x)\right)^{2}-2st\left[u^{+}(x) u^{-}(y)+u^{-}(x) u^{+}(y)\right]\right] \\
= & \int_{V} \left|\nabla (su^{+})\right|^2 d \mu+\int_{V}\left|\nabla (tu^{-})\right|^2 d \mu-stK_{V}(u).
\end{aligned}
$$

(ii) By direct computation, we get that
$$
\begin{aligned}
&\int_{V} \nabla\left(su^{+}+tu^{-}\right)\nabla (su^+)\, d \mu \\= & \frac{1}{2} \sum_{x \in V} \sum_{y \sim x}\left[\left(su^{+}+tu^{-}\right)(y)-\left(su^{+}+tu^{-}\right)(x)\right]\left[su^{+}(y)-su^{+}(x)\right] \\
= & \frac{1}{2} \sum_{x \in V} \sum_{y \sim x}\left[\left(su^{+}(y)-su^+(x)\right)^{2}-st\left[u^{+}(x) u^{-}(y)+u^{-}(x) u^{+}(y)\right]\right] \\
= & \int_{V} \left|\nabla(s u^{+})\right|^2 d \mu-\frac{st}{2} K_{V}(u) .
\end{aligned}
$$

(iii) The proof is similar to that of (ii), we omit here.

\end{proof}

\begin{crl}  Let $s,t>0$. Then for any $u \in \mathcal{H}_\lambda$, we have
\begin{itemize}
    \item[(i)] \begin{equation}\label{7.1}
\begin{aligned}
J_{\lambda}(su^++tu^-)=&J_{\lambda}\left(su^{+}\right)+J_{\lambda}\left(tu^{-}\right)-\frac{a}{2}stK_{V}(u)+\frac{b}{4}s^2t^2K^2_V(u)+\frac{b}{2}\left\|\nabla (su^{+})\right\|_{2}^{2}\left\|\nabla (tu^{-})\right\|_{2}^{2}\\ &-\frac{b}{2}stK_V(u)\left(\left\|\nabla (su^{+})\right\|_{2}^{2}+\left\|\nabla (tu^{-})\right\|_{2}^{2}\right).\\
\end{aligned}
\end{equation}

\item[(ii)] \begin{equation}\label{2.5}
\begin{aligned}
(J_{\lambda}^{\prime}(su^++tu^-),su^{+}) =&(J_{\lambda}^{\prime}\left(su^{+}\right),su^{+})-\frac{a}{2}stK_{V}(u)+\frac{b}{2}s^2t^2K^2_V(u)+b\left\|\nabla (su^{+})\right\|_{2}^{2}\left\|\nabla (tu^{-})\right\|_{2}^{2}\\&-\frac{b}{2}stK_V(u)\left(\left\|\nabla (su^{+})\right\|_{2}^{2}+\left\|\nabla (tu^{-})\right\|_{2}^{2}\right)
-bstK_V(u)\|\nabla (su^{+})\|^2_2.
\end{aligned}
\end{equation}

\item[(iii)] \begin{equation}\label{2.6}
\begin{aligned}
(J_{\lambda}^{\prime}(su^++tu^-),tu^{-}) &=(J_{\lambda}^{\prime}\left(tu^{-}\right),tu^{-})-\frac{a}{2}stK_{V}(u)+\frac{b}{2}s^2t^2K^2_V(u)+b\left\|\nabla (su^{+})\right\|_{2}^{2}\left\|\nabla (tu^{-})\right\|_{2}^{2}\\&-\frac{b}{2}stK_V(u)\left(\left\|\nabla (su^{+})\right\|_{2}^{2}+\left\|\nabla (tu^{-})\right\|_{2}^{2}\right)
-bstK_V(u)\|\nabla (tu^{-})\|^2_2.
\end{aligned}
\end{equation}
\end{itemize}
\end{crl}

\begin{proof}
Since the proof of our results is similar, we just prove the equality (\ref{2.5}).
By (i) and (ii) of Proposition \ref{o}, we obtain that
$$
\begin{aligned}
&(J_{\lambda}^{\prime}(su^++tu^-),su^{+})\\=&
a\int_{V} \nabla (su^++tu^-)\nabla (su^{+})\,d\mu+\int_{V}h_\lambda(x)(su^++tu^-)(su^{+})\,d \mu\\&+b \int_{V}|\nabla(su^++tu^-)|^{2}\,d \mu \int_{V} \nabla (su^++tu^-) \nabla (su^+)\,d \mu\\&-\int_{V}|(su^++tu^-)|^{p-2}(su^++tu^-)(su^+)\log (su^++tu^-)^2\,d\mu\\=& a\left(\int_{V} \left|\nabla(s u^{+})\right|^2 d \mu-\frac{st}{2} K_{V}(u)\right)+\int_{V}h_\lambda(x)(su^+)^2\,d \mu\\&+ b\left(\int_{V} \left|\nabla (su^{+})\right|^2 d \mu+\int_{V}\left|\nabla (tu^{-})\right|^2 d \mu-stK_{V}(u)\right)\left(\int_{V} \left|\nabla(s u^{+})\right|^2 d \mu-\frac{st}{2} K_{V}(u)\right)\\&-\int_{V}|su^+|^{p}\log (su^+)^2\,d\mu \\=& \int_{V}  a|\nabla(s u^{+})|^2 +h_\lambda(x)(su^+)^2\,d \mu  +b\left(\int_{V} \left|\nabla (su^{+})\right|^2 d \mu\right)^2-\int_{V}|su^+|^{p}\log (su^+)^2\,d\mu\\&-\frac{a}{2}stK_{V}(u)+\frac{b}{2}s^2t^2K^2_V(u)+b\left\|\nabla (su^{+})\right\|_{2}^{2}\left\|\nabla (tu^{-})\right\|_{2}^{2}\\&-\frac{b}{2}stK_V(u)\left(\left\|\nabla (su^{+})\right\|_{2}^{2}+\left\|\nabla (tu^{-})\right\|_{2}^{2}\right)
-bstK_V(u)\|\nabla (su^{+})\|^2_2\\=&(J_{\lambda}^{\prime}\left(su^{+}\right),su^{+})-\frac{a}{2}stK_{V}(u)+\frac{b}{2}s^2t^2K^2_V(u)+b\left\|\nabla (su^{+})\right\|_{2}^{2}\left\|\nabla (tu^{-})\right\|_{2}^{2}\\&-\frac{b}{2}stK_V(u)\left(\left\|\nabla (su^{+})\right\|_{2}^{2}+\left\|\nabla (tu^{-})\right\|_{2}^{2}\right)
-bstK_V(u)\|\nabla (su^{+})\|^2_2.
\end{aligned}
$$
\end{proof}

\begin{rem}
 Let $s=t=1$ in (\ref{7.1}), (\ref{2.5}) and (\ref{2.6}). Then for any $u\in\mathcal{H}_\lambda\backslash\{0\}$, we have
 \begin{itemize}
 \item[(i)] The following results highlight a notable distinction
between the discrete and continuous cases, 
$$J_\lambda(u)\neq J_\lambda(u^+)+J_\lambda(u^-),\qquad (J'_\lambda(u),u)\neq (J'_\lambda(u^+),u^+)+(J'_\lambda(u^-),u^-).$$

\item[(ii)]  A direct calculation yields that
\begin{equation}\label{0.3}
(J'_\lambda(u),u)=(J'_\lambda(u),u^+)+(J'_\lambda(u),u^-).
\end{equation}
\end{itemize}

\end{rem}

By similar arguments as above, for any $u \in \mathcal{H}_{0}^{1}(\Omega)$, we have similar results.
\begin{crl} Let $s,t>0$. Then for any $u \in \mathcal{H}_{0}^{1}(\Omega)$, we have
    \begin{itemize}
        \item[(i)]$$
\begin{aligned}
J_{\Omega}(su^++tu^-)=&J_{\Omega}\left(su^{+}\right)+J_{\Omega}\left(tu^{-}\right)-\frac{a}{2}stK_{\Omega}(u)+\frac{b}{4}s^2t^2K^2_\Omega(u)\\&+\frac{b}{2}\left\|\nabla (su^{+})\right\|_{\ell^2(\Omega\cup\partial \Omega)}^{2}\left\|\nabla (tu^{-})\right\|_{\ell^2(\Omega\cup\partial \Omega)}^{2}\\ &-\frac{b}{2}stK_\Omega(u)\left(\left\|\nabla (su^{+})\right\|_{\ell^2(\Omega\cup\partial \Omega)}^{2}+\left\|\nabla (tu^{-})\right\|_{\ell^2(\Omega\cup\partial \Omega)}^{2}\right),\\
\end{aligned}
$$
where $K_{\Omega}(u):=\sum\limits_{x \in \Omega \cup \partial\Omega} \sum\limits_{y \sim x} \left[u^{+}(x) u^{-}(y)+u^{-}(x) u^{+}(y)\right]\leq 0$.
\item[(ii)] $$
\begin{aligned}
(J_{\Omega}^{\prime}(su^++tu^-),su^{+}) =&(J_{\Omega}^{\prime}\left(su^{+}\right),su^{+})-\frac{a}{2}stK_{\Omega}(u)+\frac{b}{2}s^2t^2K^2_\Omega(u)\\&+b\left\|\nabla (su^{+})\right\|_{\ell^2(\Omega\cup\partial \Omega)}^{2}\left\|\nabla (tu^{-})\right\|_{\ell^2(\Omega\cup\partial \Omega)}^{2}\\&-\frac{b}{2}stK_\Omega(u)\left(\left\|\nabla (su^{+})\right\|_{\ell^2(\Omega\cup\partial \Omega)}^{2}+\left\|\nabla (tu^{-})\right\|_{\ell^2(\Omega\cup\partial \Omega)}^{2}\right)
\\&-bstK_\Omega(u)\|\nabla (su^{+})\|^2_{\ell^2(\Omega\cup\partial \Omega)}.
\end{aligned}
$$
\item[(iii)] $$
\begin{aligned}
(J_{\Omega}^{\prime}(su^++tu^-),tu^{-}) =&(J_{\Omega}^{\prime}\left(tu^{-}\right),tu^{-})-\frac{a}{2}stK_{\Omega}(u)+\frac{b}{2}s^2t^2K^2_\Omega(u)\\&+b\left\|\nabla (su^{+})\right\|_{\ell^2(\Omega\cup\partial \Omega)}^{2}\left\|\nabla (tu^{-})\right\|_{\ell^2(\Omega\cup\partial \Omega)}^{2}\\&-\frac{b}{2}stK_\Omega(u)\left(\left\|\nabla (su^{+})\right\|_{\ell^2(\Omega\cup\partial \Omega)}^{2}+\left\|\nabla (tu^{-})\right\|_{\ell^2(\Omega\cup\partial \Omega)}^{2}\right)
\\&-bstK_\Omega(u)\|\nabla (tu^{-})\|^2_{\ell^2(\Omega\cup\partial \Omega)}.
\end{aligned}
$$
    \end{itemize}
\end{crl}

\begin{lm}\label{l2.1}
  If $u \in \mathcal{H}_\lambda$ is a weak solution of the equation (\ref{0.2}), then $u$ is a pointwise solution of the equation (\ref{0.2}).  
\end{lm}

\begin{proof}
If $u \in \mathcal{H}_\lambda$ is a weak solution of the equation (\ref{0.2}), then for any $\varphi \in \mathcal{H}_\lambda$, there holds
$$
\int_{V}(a \nabla u \nabla \phi+h_\lambda(x) u \phi)\,d \mu+b \int_{V}|\nabla u|^{2}\,d \mu \int_{V} \nabla u \nabla \phi\,d \mu=\int_{V}|u|^{p-2}u\phi\log u^2\,d\mu.
$$
Since $C_{c}(V)$ is dense in $\mathcal{H}_\lambda$, for any $\varphi \in C_{c}(V)$, by integration by parts, we have
\begin{equation}\label{2.1}
-\left(a+b \int_{V}|\nabla u|^{2} d \mu\right) \int_{V}\Delta u\,\phi\,d\mu+\int_{V}h_\lambda(x) u \varphi d \mu=\int_{V}|u|^{p-2}u\varphi \log u^{2} d \mu.
\end{equation}
For any fixed $x_0\in V$, let
$$
\varphi_0(x)= \begin{cases}1, & x=x_0 \\ 0, & x \neq x_0.\end{cases}
$$
Clearly, $\varphi_0 \in C_c(V)$. Let $\phi_0$ be a test function in (\ref{2.1}), we get that
$$-\left(a+b \int_{V}|\nabla u(x_0)|^{2} d \mu\right)\Delta u(x_0)+ h_\lambda (x_0)        u(x_0)=|u(x_0)|^{p-2}u(x_0)\log u^{2}(x_0).$$ By the arbitrariness of $x_0$, we conclude that $u$ is a pointwise solution of the equation (\ref{0.2}).

\end{proof}

\begin{lm}
 If $u \in H_{0}^{1}(\Omega)$ is a weak solution of the equation (\ref{1.8}), then $u$ is a pointwise solution of the equation (\ref{1.8}).   
\end{lm}
\begin{proof}
The proof is similar to that of Lemma \ref{l2.1}, we omit here.
\end{proof}

\begin{lm}\label{l2}
 Let $(h_1)$-$(h_2)$ hold. Then there exists a constant $\lambda_{0}>0$ such that, for any $\lambda \geq \lambda_{0}$, the space $\mathcal{H}_{\lambda}$ is continuously embedded into $\ell^{q}(V)$ with $q\in[2,\infty].$ Moreover, for any bounded sequence $\{u_k\}\subset \mathcal{H}_\lambda$, there exists $u_\lambda\in \mathcal{H}_\lambda$ such that, up to a subsequence,
 $$
\begin{cases}u_k \rightharpoonup u_{\lambda}, & \text { weakly in } \mathcal{H}_{\lambda}, \\ u_k \rightarrow u_{\lambda}, & \text { pointwise in } V, \\ u_k \rightarrow u_{\lambda}, & \text { strongly in } \ell^{q}(V),~q \in[2,\infty].\end{cases}
$$
\end{lm}

\begin{proof}
Clearly, for any $\lambda>0$, we have
$\|u\|_2\leq \|u\|_{\mathcal{H}_{\lambda}}.$
Moreover, for any $p\leq q$, we have 
$\|u\|_q\leq \|u\|_p.$
Hence the above two inequalities imply that for any $2\leq q\leq \infty$,
$$\|u\|_q\leq \|u\|_{\mathcal{H}_{\lambda}}.$$
This means that $\mathcal{H}_{\lambda}$ is continuously embedded into $\ell^{q}(V)$ with $q\in[2,\infty].$

 Let $\left\{u_k\right\}$ be bounded in $\mathcal{H}_{\lambda}$. Hence there exists $u_\lambda\in\mathcal{H}_{\lambda}$ such that, up to a subsequence, $$u_k \rightharpoonup u_\lambda,\quad \text{in}~ \mathcal{H}_{\lambda}.$$ 
 
 Since $\left\{u_k\right\}$ is bounded in $\mathcal{H}_{\lambda}$, and hence bounded in $\ell^{\infty}(V)$. By diagonal principle, there exists a subsequence of $\left\{u_k\right\}$ (still denoted itself) such that
 $$u_k \rightarrow u_{\lambda},\quad \text { pointwise in } V.$$

We now prove that there exists a constant $\lambda_{0}>0$ such that, for any $\lambda \geq \lambda_{0}$, $u_k \rightarrow u$ in $\ell^{q}(V)$ with $2 \leq q \leq \infty$. Since $\{u_k\}$ is bounded in $\mathcal{H}_{\lambda}$ and $u_\lambda \in \mathcal{H}_{\lambda}$, there exists $C>0$ such that
$$
\int_{V}\lambda  h(x)\left(u_k-u_\lambda\right)^{2} d \mu \leq C.
$$

We claim that, up to a subsequence,
$$
\lim _{k\rightarrow\infty} \|u_k-u_\lambda\|_2=0.
$$
In fact, by $\left(h_{2}\right)$, we get that
$$
\begin{aligned}
\int_{V}\left(u_k-u_\lambda\right)^{2} d \mu & =\int_{D_{M}}\left(u_k-u_\lambda\right)^{2} d \mu+\int_{V \backslash D_{M}}\left(u_k-u_\lambda\right)^{2} d \mu \\
& \leq \int_{D_{M}}\left(u_k-u_\lambda\right)^{2} d \mu+\int_{V \backslash D_{M}} \frac{1}{\lambda M} \lambda h(x)\left(u_k-u_\lambda\right)^{2} d \mu \\
& \leq \int_{D_{M}}\left(u_k-u_\lambda\right)^{2} d \mu+\frac{C}{\lambda M},
\end{aligned}
$$
where $D_M=\{x\in V: h(x)<M\}$.
For any $\varepsilon>0$, there exists $\lambda_{0}>0$ such that, for any $\lambda\geq\lambda_{0}$,  $\frac{C}{\lambda M}<\varepsilon$. Moreover, since $D_M$ is a finite set, up to a subsequence, we have
$$
\lim _{k\rightarrow\infty} \int_{D_{M}}\left(u_k-u_\lambda\right)^{2} d \mu=0.
$$
Hence the claim holds. Then $$\left\|u_{k}-u\right\|_{\infty} \leq \|u_k-u_\lambda\|_2\rightarrow 0,\quad k\rightarrow\infty.$$ 

For any $2<q<\infty$, by interpolation inequality, we get that
$$
\|u_k-u_\lambda\|^q_q\leq \|u_k-u_\lambda\|^2_2\|u_k-u_\lambda\|^{q-2}_\infty\rightarrow 0,\quad k\rightarrow\infty.
$$
Therefore, up to a subsequence, $u_k \rightarrow u_\lambda$ in $\ell^{q}(V)$ with $2 \leq q \leq\infty$.

\end{proof}

\section{Existence of least energy sign-changing solutions}
In this section, we prove Theorem \ref{t1} and Theorem \ref{t2} by the Nehari manifold method and the Miranda's theorem. We start our analysis with a few auxiliary lemmas that useful in our theorems.

\begin{lm}\label{l3}
 For any $u \in \mathcal{M}_{\lambda}$ and $(s, t) \in(0, \infty) \times(0, \infty)$ with $(s, t) \neq(1,1)$, we have
$$
J_{\lambda}(u) > J_{\lambda}\left(s u^{+}+t u^{-}\right).
$$

\end{lm}

\begin{proof}

For any $u \in \mathcal{M}_{\lambda}$ and $s, t>0$, by direct calculation, we obtain that
$$
\begin{aligned}
& J_\lambda(u)-J_\lambda\left(s u^{+}+t u^{-}\right) \\
= & \frac{1}{2}\left(\|u\|_{\mathcal{H}_\lambda}^{2}-\left\|s u^{+}+t u^{-}\right\|_{\mathcal{H}_\lambda}^{2}\right)+\frac{b}{4}\left(\|\nabla u\|_{2}^{4}-\left\|\nabla\left(s u^{+}+t u^{-}\right)\right\|_{2}^{4}\right) \\
& +\frac{2}{p^{2}} \int_{V}\left[|u|^{p}-\left|s u^{+}+t u^{-}\right|^{p}\right] \,d\mu-\frac{1}{p} \int_{V}\left[|u|^{p} \log u^{2}-\left|s u^{+}+t u^{-}\right|^{p} \log \left(s u^{+}+t u^{-}\right)^{2}\right] \,d\mu.
\\
= & \frac{1-s^{2}}{2}\left\|u^{+}\right\|_{\mathcal{H}_\lambda}^{2}+\frac{1-t^{2}}{2}\left\|u^{-}\right\|_{\mathcal{H}_\lambda}^{2}-\frac{a(1-st)}{2}K_V(u)+\frac{b\left(1-s^{4}\right)}{4}\left\|\nabla u^{+}\right\|_{2}^{4}+\frac{b\left(1-t^{4}\right)}{4}\left\|\nabla u^{-}\right\|_{2}^{4} \\
& +\frac{b\left(1-s^{2} t^{2}\right)}{2}\left\|\nabla u^{+}\right\|_{2}^{2}\left\|\nabla u^{-}\right\|_{2}^{2}-\frac{b(1-s^3t)}{2}K_V(u)\left\|\nabla u^{+}\right\|_{2}^{2}-\frac{b(1-st^3)}{2}K_V(u)\left\|\nabla u^{-}\right\|_{2}^{2}\\&+\frac{b(1-s^2t^2)}{4}K^2_V(u)+\frac{2}{p^{2}} \int_{V}\left[\left|u^{+}\right|^{p}-\left|s u^{+}\right|^{p}+\left|u^{-}\right|^{p}-\left|t u^{-}\right|^{p}\right] \,d\mu\\
& -\frac{1}{p} \int_{V}\left[\left|u^{+}\right|^{p} \log \left(u^{+}\right)^{2}-\left|s u^{+}\right|^{p} \log \left(u^{+}\right)^{2}-\left|s u^{+}\right|^{p} \log s^{2}\right] \,d\mu\\
& -\frac{1}{p} \int_{V}\left[\left|u^{-}\right|^{p} \log \left(u^{-}\right)^{2}-\left|t u^{-}\right|^{p} \log \left(u^{-}\right)^{2}-\left|t u^{-}\right|^{p} \log t^{2}\right] \,d\mu\\
= & \frac{1-s^{p}}{p}\left( J_{\lambda}^{\prime}(u), u^{+}\right)+\frac{1-t^{p}}{p}\left(J_{\lambda}^{\prime}(u), u^{-}\right)+\left(\frac{1-s^{2}}{2}-\frac{1-s^{p}}{p}\right)\left\|u^{+}\right\|_{\mathcal{H}_\lambda}^{2} \\
& +\left(\frac{1-t^{2}}{2}-\frac{1-t^{p}}{p}\right)\left\|u^{-}\right\|_{\mathcal{H}_\lambda}^{2}+b\left[\left(\frac{1-s^{4}}{4}-\frac{1-s^{p}}{p}\right)\left\|\nabla u^{+}\right\|_{2}^{4}\right. \\
& \left.+\left(\frac{1-t^{4}}{4}-\frac{1-t^{p}}{p}\right)\left\|\nabla u^{-}\right\|_{2}^{4}+\left(\frac{1-s^{2} t^{2}}{2}-\frac{1-s^{p}}{p}-\frac{1-t^{p}}{p}\right)\left\|\nabla u^{+}\right\|_{2}^{2}\left\|\nabla u^{-}\right\|_{2}^{2}\right]\\&+\frac{b}{2}\left(\frac{1-s^{2} t^{2}}{2}-\frac{1-s^{p}}{p}-\frac{1-t^{p}}{p}\right)K_V^2(u)-\frac{a}{2}\left((1-st)-\frac{1-s^{p}}{p}-\frac{1-t^{p}}{p}\right)K_V(u)\\&-\frac{b}{2}\left((1-s^3t)-\frac{3(1-s^{p})}{p}-\frac{1-t^{p}}{p}\right)K_V(u)\|\nabla u^+\|^2_2\\&-\frac{b}{2}\left((1-st^3)-\frac{1-s^{p}}{p}-\frac{3(1-t^{p})}{p}\right)K_V(u)\|\nabla u^-\|^2_2\\
& +\frac{2\left(1-s^{p}\right)+p s^{p} \log s^{2}}{p^{2}} \int_{V}\left|u^{+}\right|^{p} \,d\mu+\frac{2\left(1-t^{p}\right)+p t^{p} \log t^{2}}{p^{2}} \int_{V}\left|u^{-}\right|^{p} \,d\mu
\\
= & \left(\frac{1-s^{2}}{2}-\frac{1-s^{p}}{p}\right)\left\|u^{+}\right\|_{\mathcal{H}_\lambda}^{2}+\left(\frac{1-t^{2}}{2}-\frac{1-t^{p}}{p}\right)\left\|u^{-}\right\|_{\mathcal{H}_\lambda}^{2}+b\left[\left(\frac{1-s^{4}}{4}-\frac{1-s^{p}}{p}\right)\left\|\nabla u^{+}\right\|_{2}^{4}\right. \\
& \left.+\left(\frac{1-t^{4}}{4}-\frac{1-t^{p}}{p}\right)\left\|\nabla u^{-}\right\|_{2}^{4}+\left(\frac{1-s^{2} t^{2}}{2}-\frac{1-s^{p}}{p}-\frac{1-t^{p}}{p}\right)\left\|\nabla u^{+}\right\|_{2}^{2}\left\|\nabla u^{-}\right\|_{2}^{2}\right]\\&+\frac{b}{2}\left(\frac{1-s^{2} t^{2}}{2}-\frac{1-s^{p}}{p}-\frac{1-t^{p}}{p}\right)K_V^2(u)-\frac{a}{2}\left((1-st)-\frac{1-s^{p}}{p}-\frac{1-t^{p}}{p}\right)K_V(u)\\&-\frac{b}{2}\left((1-s^3t)-\frac{3(1-s^{p})}{p}-\frac{1-t^{p}}{p}\right)K_V(u)\|\nabla u^+\|^2_2\\&-\frac{b}{2}\left((1-st^3)-\frac{1-s^{p}}{p}-\frac{3(1-t^{p})}{p}\right)K_V(u)\|\nabla u^-\|^2_2 \\
& +\frac{2\left(1-s^{p}\right)+p s^{p} \log s^{2}}{p^{2}} \int_{V}\left|u^{+}\right|^{p} \,d\mu+\frac{2\left(1-t^{p}\right)+p t^{p} \log t^{2}}{p^{2}} \int_{V}\left|u^{-}\right|^{p} \,d\mu.
\end{aligned}
$$
Let 
$$f(x)=2\left(1-x^{p}\right)+p x^{p} \log x^{2}, \quad x \in(0,\infty).$$ 
A direct calculation yields that $f^{\prime}(x)=$ $p^{2} x^{p-1} \log x^{2}$. Clearly, for $x \in(0,1)$, $f^{\prime}(x)<0$, and for $x \in(1,\infty), f^{\prime}(x)>0$. Hence for any $x \in(0,1) \cup(1,\infty)$, 
\begin{equation}\label{3.2}
f(x)=2\left(1-x^{p}\right)+p x^{p} \log x^{2}>0.
\end{equation}
Moreover, let $$g(x)=\frac{1-a^{x}}{x},\quad x\in (0,\infty).$$  One gets easily that for $a>0$ with $a \neq 1$, $g(x)$ is a strictly decreasing function in $(0,\infty)$. Therefore, for $(s, t) \in(0, \infty) \times(0, \infty)$  and $(s, t) \neq(1,1)$, we have $$ 
\begin{aligned}
&J_\lambda(u)-J_\lambda\left(s u^{+}+t u^{-}\right)\\ >& b\left(\frac{1-s^{2} t^{2}}{2}-\frac{1-s^{p}}{p}-\frac{1-t^{p}}{p}\right)\left\|\nabla u^{+}\right\|_{2}^{2}\left\|\nabla u^{-}\right\|_{2}^{2} \\&+\frac{b}{2}\left(\frac{1-s^{2} t^{2}}{2}-\frac{1-s^{p}}{p}-\frac{1-t^{p}}{p}\right)K_V^2(u)-\frac{a}{2}\left((1-st)-\frac{1-s^{p}}{p}-\frac{1-t^{p}}{p}\right)K_V(u)\\&-\frac{b}{2}\left((1-s^3t)-\frac{3(1-s^{p})}{p}-\frac{1-t^{p}}{p}\right)K_V(u)\|\nabla u^+\|^2_2\\&-\frac{b}{2}\left((1-st^3)-\frac{1-s^{p}}{p}-\frac{3(1-t^{p})}{p}\right)K_V(u)\|\nabla u^-\|^2_2\\
 =& b\left[\frac{\left(s^{2}-t^{2}\right)^{2}}{4}+\left(\frac{1-s^{4}}{4}-\frac{1-s^{p}}{p}\right)+\left(\frac{1-t^{4}}{4}-\frac{1-t^{p}}{p}\right)\right]\left\|\nabla u^{+}\right\|_{2}^{2}\left\|\nabla u^{-}\right\|_{2}^{2}\\&+\frac{b}{2}\left[\frac{\left(s^{2}-t^{2}\right)^{2}}{4}+\left(\frac{1-s^{4}}{4}-\frac{1-s^{p}}{p}\right)+\left(\frac{1-t^{4}}{4}-\frac{1-t^{p}}{p}\right)\right]K_V^2(u)\\&-\frac{a}{2}\left[\frac{\left(s-t\right)^{2}}{2}+\left(\frac{1-s^{2}}{2}-\frac{1-s^{p}}{p}\right)+\left(\frac{1-t^{2}}{2}-\frac{1-t^{p}}{p}\right)\right]K_V(u)\\&-\frac{b}{2}\left[\frac{\left(s^{3}-t\right)^{2}}{2}+\left(\frac{3(1-s^{6})}{6}-\frac{3(1-s^{p})}{p}\right)+\left(\frac{1-t^{2}}{2}-\frac{1-t^{p}}{p}\right)\right]K_V(u)\|\nabla u^+\|^2_2\\&-\frac{b}{2}\left[\frac{\left(s-t^{3}\right)^{2}}{2}+\left(\frac{1-s^{2}}{2}-\frac{1-s^{p}}{p}\right)+\left(\frac{3(1-t^{6})}{6}-\frac{3(1-t^{p})}{p}\right)\right]K_V(u)\|\nabla u^-\|^2_2\\>&0,
\end{aligned}
$$
where we have used the facts that $p>6$ and $K_V(u)<0$ for $u\in \mathcal{M}_\lambda$.
The proof is completed.

\end{proof}

\begin{lm}\label{4}
    
If $u \in \mathcal{H}_\lambda$ with $u^{ \pm} \neq 0$, then there exists a unique positive number pair $\left(s_{u}, t_{u}\right)$ such that $s_{u} u^{+}+t_{u} u^{-} \in \mathcal{M}_{\lambda}$.
\end{lm}

\begin{proof}
    
For $s, t>0$, let
\begin{eqnarray*}
 g_{1}(s, t) = \left(J_{\lambda}^{\prime}\left(s u^{+}+t u^{-}\right),su^{+}\right),\qquad g_{2}(s, t) =\left(J_{\lambda}^{\prime}\left(s u^{+}+t u^{-}\right),t u^{-}\right).  
\end{eqnarray*}
By (\ref{2.5}) and (\ref{2.6}), we get respectively that
\begin{equation}\label{2.7}
\begin{aligned}
g_1(s,t) =&(J_{\lambda}^{\prime}(su^++tu^-),su^{+})\\=&(J_{\lambda}^{\prime}\left(su^{+}\right),su^{+})-\frac{a}{2}stK_{V}(u)+\frac{b}{2}s^2t^2K^2_V(u)+b\left\|\nabla (su^{+})\right\|_{2}^{2}\left\|\nabla (tu^{-})\right\|_{2}^{2}\\&-\frac{b}{2}stK_V(u)\left(\left\|\nabla (su^{+})\right\|_{2}^{2}+\left\|\nabla (tu^{-})\right\|_{2}^{2}\right)
-bstK_V(u)\|\nabla (su^{+})\|^2_2\\=& s^{2}\|u^{+}\|_{\mathcal{H}_\lambda}^{2}+b s^{4}\left\|\nabla u^{+}\right\|_{2}^{4}-\int_{V}\left|s u^{+}\right|^{p} \log \left(s u^{+}\right)^{2}\,d\mu-\frac{a}{2}stK_{V}(u)+\frac{b}{2}s^2t^2K^2_V(u) 
\\&+b s^{2} t^{2}\left\|\nabla u^{+}\right\|_{2}^{2}\left\|\nabla u^{-}\right\|_{2}^{2}-\frac{b}{2}stK_V(u)\left(s^2\left\|\nabla u^{+}\right\|_{2}^{2}+t^2\left\|\nabla u^{-}\right\|_{2}^{2}\right)
-bs^3tK_V(u)\|\nabla u^{+}\|^2_2,
\end{aligned}
\end{equation}
and
\begin{equation}\label{2.8}
\begin{aligned}
g_2(s,t) =&(J_{\lambda}^{\prime}(su^++tu^-),tu^{-})\\=&(J_{\lambda}^{\prime}\left(tu^{-}\right),tu^{-})-\frac{a}{2}stK_{V}(u)+\frac{b}{2}s^2t^2K^2_V(u)+b\left\|\nabla (su^{+})\right\|_{2}^{2}\left\|\nabla (tu^{-})\right\|_{2}^{2}\\&-\frac{b}{2}stK_V(u)\left(\left\|\nabla (su^{+})\right\|_{2}^{2}+\left\|\nabla (tu^{-})\right\|_{2}^{2}\right)
-bstK_V(u)\|\nabla (tu^{-})\|^2_2\\=& t^{2}\|u^{-}\|_{\mathcal{H}_\lambda}^{2}+b t^{4}\left\|\nabla u^{-}\right\|_{2}^{4}-\int_{V}\left|t u^{-}\right|^{p} \log \left(t u^{-}\right)^{2}\,d\mu-\frac{a}{2}stK_{V}(u)+\frac{b}{2}s^2t^2K^2_V(u) 
\\&+b s^{2} t^{2}\left\|\nabla u^{+}\right\|_{2}^{2}\left\|\nabla u^{-}\right\|_{2}^{2}-\frac{b}{2}stK_V(u)\left(s^2\left\|\nabla u^{+}\right\|_{2}^{2}+t^2\left\|\nabla u^{-}\right\|_{2}^{2}\right)
-bst^3K_V(u)\|\nabla u^{-}\|^2_2.
\end{aligned}
\end{equation}
Let $t=s$ in (\ref{2.7}) and (\ref{2.8}), then
$$
\begin{aligned}
g_1(s,s) =& s^{2}\|u^{+}\|_{\mathcal{H}_\lambda}^{2}+b s^{4}\left\|\nabla u^{+}\right\|_{2}^{4}-s^p\int_{V}\left| u^{+}\right|^{p} \log \left( u^{+}\right)^{2}\,d\mu-s^p\log s^2\int_{V}\left|u^{+}\right|^{p}\,d\mu-\frac{a}{2}s^2K_{V}(u)\\&+\frac{b}{2}s^4K^2_V(u) 
+b s^{4} \left\|\nabla u^{+}\right\|_{2}^{2}\left\|\nabla u^{-}\right\|_{2}^{2}-\frac{b}{2}s^4K_V(u)\left(\left\|\nabla u^{+}\right\|_{2}^{2}+\left\|\nabla u^{-}\right\|_{2}^{2}\right)
-bs^4K_V(u)\|\nabla u^{+}\|^2_2,
\end{aligned}
$$
and
$$
\begin{aligned}
g_2(s,s) =& s^{2}\|u^{-}\|_{\mathcal{H}_\lambda}^{2}+b s^{4}\left\|\nabla u^{-}\right\|_{2}^{4}-s^p\int_{V}\left| u^{-}\right|^{p} \log \left( u^{-}\right)^{2}\,d\mu-s^p\log s^2\int_{V}\left|u^{-}\right|^{p}\,d\mu-\frac{a}{2}s^2K_{V}(u)\\&+\frac{b}{2}s^4K^2_V(u) 
+b s^{4} \left\|\nabla u^{+}\right\|_{2}^{2}\left\|\nabla u^{-}\right\|_{2}^{2}-\frac{b}{2}s^4K_V(u)\left(\left\|\nabla u^{+}\right\|_{2}^{2}+\left\|\nabla u^{-}\right\|_{2}^{2}\right)
-bs^4K_V(u)\|\nabla u^{-}\|^2_2.
\end{aligned}
$$
Since $p>6$ and $K_V(u)< 0$, one gets easily that $ g_{1}(s, s)>0$ and $g_{2}(s, s)>0$ for $s>0$ small enough and $g_{1}(s, s)<0$ and $g_{2}(s, s)<0$ for $s>0$ large enough. Therefore, there exist $r$ and $R$ with $0<r<R$ such that
\begin{equation}\label{0.5}
g_{1}(r, r)>0,\quad g_{2}(r, r)>0,\quad g_{1}(R, R)<0,\quad g_{2}(R, R)<0.
\end{equation}
By (\ref{2.7}), (\ref{2.8}) and (\ref{0.5}), we obtain that
$$
\begin{aligned}
& g_{1}(r, t)>0, \qquad g_{1}(R, t)<0,\quad t \in[r, R], \\
& g_{2}(s, r)>0,\qquad g_{2}(s, R)<0,\quad s \in[r, R].
\end{aligned}
$$
By the Miranda's theorem \cite{K}, there exist $s_{u}, t_{u} \in(r, R)$ such that $g_{1}\left(s_{u}, t_{u}\right)=g_{2}\left(s_{u}, t_{u}\right)=0$, which implies that $s_{u} u^{+}+t_{u} u^{-} \in \mathcal{M}_{\lambda}$.

Next, we prove the uniqueness of the pair $\left(s_{u}, t_{u}\right)$. In fact, by contradiction, suppose there exist $\left(s_{1}, t_{1}\right)$ and $\left(s_{2}, t_{2}\right)$ with $s_1\neq s_2$ and $t_1\neq t_2$ such that $s_{i} u^{+}+t_{i} u^{-} \in \mathcal{M}_{\lambda}$ with $i=1, 2$.

Let $s=\frac{s_2}{s_1}$ and $t=\frac{t_2}{t_1}$. Clearly $s\neq 1$ and $t\neq 1$. Then by Lemma \ref{l3}, we have
\begin{equation}\label{3.1}
J_\lambda(s_2u^++t_2u^-) =J_\lambda \left(s(s_1u^+)+t(t_1u^-)\right)<J_\lambda (s_1u^++t_1u^-).
\end{equation}
Similarly, we have
\begin{equation}\label{3.2}
  J_\lambda(s_1u^++t_1u^-) =J_\lambda \left(\frac{1}{s}(s_2u^+)+\frac{1}{t}(t_2u^-)\right)<J_\lambda (s_2u^++t_2u^-).  
\end{equation}
Clearly, (\ref{3.1}) contradicts (\ref{3.2}). Hence there exists a unique positive number pair $(s_u, t_u)$ such that $s_uu^++t_uu^-\in\mathcal{M}_\lambda.$

\end{proof}

\begin{lm}\label{l5}
Let $u \in \mathcal{H}_\lambda$ with $u^{ \pm} \neq 0$ such that $(J_{\lambda}^{\prime}(u),u^{ \pm}) \leq 0$. Then the unique pair $\left(s_{u}, t_{u}\right)$ obtained in Lemma \ref{4} satisfies $s_{u}, t_{u} \in(0,1]$. In particular, the $"="$ holds if and only if $s_u=t_u=1.$
\end{lm}

\begin{proof}
By Lemma \ref{4}, there exists a unique positive number pair $(s_u,t_u)$ such that $s_uu^++t_uu^-\in \mathcal{M}_\lambda$. Without loss of generality, we assume that $0<t_{u} \leq s_{u}$. By (\ref{2.5}), we get that
\begin{equation}\label{3.4}
\begin{aligned}
&(J_{\lambda}^{\prime}(s_uu^++t_uu^-),s_uu^{+})\\=&s_u^{2}\|u^{+}\|_{\mathcal{H}_\lambda}^{2}+b s_u^{4}\left\|\nabla u^{+}\right\|_{2}^{4}-\int_{V}\left|s_u u^{+}\right|^{p} \log \left(s_u u^{+}\right)^{2}\,d\mu-\frac{a}{2}s_ut_uK_{V}(u)+\frac{b}{2}s_u^2t_u^2K^2_V(u) 
\\&+b s_u^{2} t_u^{2}\left\|\nabla u^{+}\right\|_{2}^{2}\left\|\nabla u^{-}\right\|_{2}^{2}-\frac{b}{2}s_ut_uK_V(u)\left(s_u^2\left\|\nabla u^{+}\right\|_{2}^{2}+t_u^2\left\|\nabla u^{-}\right\|_{2}^{2}\right)
-bs_u^3t_uK_V(u)\|\nabla u^{+}\|^2_2.
\end{aligned}
\end{equation}
Since $s_{u} u^{+}+t_{u} u^{-} \in \mathcal{M}_{\lambda}$, we have
\begin{equation}\label{3.3}
\begin{aligned}
\int_{V}\left|s_u u^{+}\right|^{p} \log \left(s_u u^{+}\right)^{2}\,d\mu=&s_u^{2}\|u^{+}\|_{\mathcal{H}_\lambda}^{2}+b s_u^{4}\left\|\nabla u^{+}\right\|_{2}^{4}-\frac{a}{2}s_ut_uK_{V}(u)+\frac{b}{2}s_u^2t_u^2K^2_V(u) 
\\&+b s_u^{2} t_u^{2}\left\|\nabla u^{+}\right\|_{2}^{2}\left\|\nabla u^{-}\right\|_{2}^{2}-\frac{b}{2}s_ut_uK_V(u)\left(s_u^2\left\|\nabla u^{+}\right\|_{2}^{2}+t_u^2\left\|\nabla u^{-}\right\|_{2}^{2}\right)
\\&-bs_u^3t_uK_V(u)\|\nabla u^{+}\|^2_2\\ \leq &s_u^{2}\|u^{+}\|_{\mathcal{H}_\lambda}^{2}+b s_u^{4}\left\|\nabla u^{+}\right\|_{2}^{4}-\frac{a}{2}s^2_uK_{V}(u)+\frac{b}{2}s_u^4K^2_V(u) 
\\&+b s_u^{4} \left\|\nabla u^{+}\right\|_{2}^{2}\left\|\nabla u^{-}\right\|_{2}^{2}-\frac{b}{2}s^4_uK_V(u)\left(\left\|\nabla u^{+}\right\|_{2}^{2}+\left\|\nabla u^{-}\right\|_{2}^{2}\right)
\\&-bs_u^4K_V(u)\|\nabla u^{+}\|^2_2,
\end{aligned}    
\end{equation}
where we have used the fact that $K_V(u)<0$.

Let $s=t=1$ in (\ref{3.4}), we get that
$$
\begin{aligned}
(J_{\lambda}^{\prime}(u),u^{+})=&\|u^{+}\|_{\mathcal{H}_\lambda}^{2}+b\left\|\nabla u^{+}\right\|_{2}^{4}-\int_{V}\left|u^{+}\right|^{p} \log \left(u^{+}\right)^{2}\,d\mu-\frac{a}{2}K_{V}(u)+\frac{b}{2}K^2_V(u) 
\\&+b\left\|\nabla u^{+}\right\|_{2}^{2}\left\|\nabla u^{-}\right\|_{2}^{2}-\frac{b}{2}K_V(u)\left(\left\|\nabla u^{+}\right\|_{2}^{2}+\left\|\nabla u^{-}\right\|_{2}^{2}\right)
-bK_V(u)\|\nabla u^{+}\|^2_2.
\end{aligned}
$$
Since $(J_{\lambda}^{\prime}(u),u^{+})\leq 0$, by the above equality, we get that
$$
\begin{aligned}
\int_{V}\left|u^{+}\right|^{p} \log \left(u^{+}\right)^{2}\,d\mu\geq &\|u^{+}\|_{\mathcal{H}_\lambda}^{2}+b\left\|\nabla u^{+}\right\|_{2}^{4}-\frac{a}{2}K_{V}(u)+\frac{b}{2}K^2_V(u) 
+b\left\|\nabla u^{+}\right\|_{2}^{2}\left\|\nabla u^{-}\right\|_{2}^{2}\\&-\frac{b}{2}K_V(u)\left(\left\|\nabla u^{+}\right\|_{2}^{2}+\left\|\nabla u^{-}\right\|_{2}^{2}\right)
-bK_V(u)\|\nabla u^{+}\|^2_2.
\end{aligned}
$$
Multiplying the previous inequality by $-s_u^p$, then
\begin{equation}\label{3.5}
 \begin{aligned}
-s_u^p\int_{V}\left|u^{+}\right|^{p} \log \left(u^{+}\right)^{2}\,d\mu\leq &-s_u^p\|u^{+}\|_{\mathcal{H}_\lambda}^{2}-bs_u^p\left\|\nabla u^{+}\right\|_{2}^{4}+\frac{a}{2}s_u^pK_{V}(u)-\frac{b}{2}s_u^pK^2_V(u) 
\\&-bs_u^p\left\|\nabla u^{+}\right\|_{2}^{2}\left\|\nabla u^{-}\right\|_{2}^{2}+\frac{b}{2}s_u^pK_V(u)\left(\left\|\nabla u^{+}\right\|_{2}^{2}+\left\|\nabla u^{-}\right\|_{2}^{2}\right)
\\&+bs_u^pK_V(u)\|\nabla u^{+}\|^2_2.
\end{aligned}   
\end{equation}
Combing (\ref{3.3}) and (\ref{3.5}), we obtain that
$$
\begin{aligned}
s_u^p\log s_u^{2}\int_{V}\left|u^{+}\right|^{p}\,d\mu\leq &(s_u^2-s_u^p)\|u^{+}\|_{\mathcal{H}_\lambda}^{2}+b(s_u^4-s_u^p)\left\|\nabla u^{+}\right\|_{2}^{4}-\frac{a}{2}(s_u^2-s_u^p)K_{V}(u)+\frac{b}{2}(s_u^4-s_u^p)K^2_V(u) 
\\&+b(s_u^4-s_u^p)\left\|\nabla u^{+}\right\|_{2}^{2}\left\|\nabla u^{-}\right\|_{2}^{2}-\frac{b}{2}(s_u^4-s_u^p)K_V(u)\left(\left\|\nabla u^{+}\right\|_{2}^{2}+\left\|\nabla u^{-}\right\|_{2}^{2}\right)
\\&-b(s_u^4-s_u^p)K_V(u)\|\nabla u^{+}\|^2_2.
\end{aligned} 
$$
If $s_u>1$, then the left hand side of the above inequality is a positive constant, while the right hand side is a negative constant. This is impossible. Hence $0<t_u\leq s_u\leq 1.$ 
\end{proof}

Similarly, we have the following results.
\begin{lm}\label{l7}
If $u \in \mathcal{H}_{0}^{1}(\Omega) $ with $u^{ \pm} \neq 0$, then there exists a unique positive number pair $\left(s_{u}, t_{u}\right)$ such that $s_{u} u^{+}+t_{u} u^{-} \in \mathcal{M}_{\Omega}$.
\end{lm}

\begin{lm}\label{l6}
 Let $u \in \mathcal{H}_{0}^{1}(\Omega)$ with $u^{ \pm} \neq 0$ such that $(J_{\Omega}^{\prime}(u),u^{ \pm}) \leq 0$. Then the unique pair $\left(s_{u}, t_{u}\right)$ obtained in Lemma \ref{l7} satisfies $s_{u}, t_{u} \in(0,1]$. In particular, the $"="$ holds if and only if $s_u=t_u=1.$
\end{lm}

Now we prove that the minimizer of $J_{\lambda}$ on $\mathcal{M}_{\lambda}$ can be achieved.

\begin{lm}\label{l9}
Let $(h_1)$-$(h_2)$ hold. Then $m_{\lambda}>0$ is achieved.
\end{lm}

\begin{proof} 
Let $\{u_k\}\subset\mathcal{M}_\lambda$ be a minimizing sequence satisfying
$$\lim\limits_{k\rightarrow\infty}J_\lambda(u_k)=m_\lambda.$$
Since $\{u_k\}\subset\mathcal{M}_\lambda$, by (\ref{0.3}), we get that $$(J'_\lambda(u_k),u_k)=(J'_\lambda(u_k),u_k^+)+(J'_\lambda(u_k),u_k^-) =0.$$
A direct calculation yields that
$$
\begin{aligned}
\lim _{k \rightarrow\infty} J_{\lambda}\left(u_{k}\right) & =\lim _{k \rightarrow\infty}\left[J_{\lambda}\left(u_{k}\right)-\frac{1}{p} (J_{\lambda}^{\prime}\left(u_{k}\right), u_{k})\right] \\
& =\lim _{k \rightarrow\infty}\left[(\frac{1}{2}-\frac{1}{p})\|u_k\|_{\mathcal{H}_\lambda}^{2}+(\frac{1}{4}-\frac{1}{p})b\left\|\nabla u_k\right\|_{2}^{4}+\frac{2}{p^2}\int_V |u_k|^p\,d\mu\right] \\
& =m_{\lambda}.
\end{aligned}
$$
Then we have
$$(\frac{1}{2}-\frac{1}{p})\|u_k\|_{\mathcal{H}_\lambda}^{2}\leq (\frac{1}{2}-\frac{1}{p})\|u_k\|_{\mathcal{H}_\lambda}^{2}+(\frac{1}{4}-\frac{1}{p})b\left\|\nabla u_k\right\|_{2}^{4}+\frac{2}{p^2}\int_V |u_k|^p\,d\mu=m_\lambda+o(1).$$
Hence $\left\{u_{k}\right\}$ is bounded in $\mathcal{H}_{\lambda}$. Then by Lemma \ref{l2}, there exists $\lambda_{0}>0$ such that, for any $\lambda \geq \lambda_{0}$, 
$$
\begin{cases}u_{k} \rightharpoonup u_{\lambda}, & \text { weakly in } \mathcal{H}_{\lambda}, \\ u_{k} \rightarrow u_{\lambda}, & \text { pointwise in } V, \\ u_{k} \rightarrow u_{\lambda}, & \text { strongly in } \ell^{q}(V), ~q \in[2,\infty],\end{cases}
$$
where $u_{\lambda} \in \mathcal{H}_{\lambda}$.
By (\ref{1.2}), for any $\varepsilon>0$ and $t\neq 0,$ we have
\begin{equation}\label{3.9}
|t|^p|\log t^2|\leq (\varepsilon |t|^{2}+C_\varepsilon |t|^{q}),\quad q>p.
\end{equation}
Moreover, note that
\begin{equation}\label{3.8}
(J'_\lambda(u_k),u_k^{\pm})=\|u_k^{\pm}\|_{\mathcal{H}_\lambda}^{2}+b\left\|\nabla u_k\right\|_{2}^{2}\left(\|\nabla u_k^{\pm}\|_{2}^{2}-\frac{1}{2}K_V(u_k)\right)-\frac{a}{2}K_{V}(u_k)-\int_{V}\left|u_k^{\pm}\right|^{p} \log \left(u_k^{\pm}\right)^{2}\,d\mu=0.
\end{equation}
Then we get that
$$
\begin{aligned}
\|u_k^{\pm}\|_{\mathcal{H}_\lambda}^{2}< &\|u_k^{\pm}\|_{\mathcal{H}_\lambda}^{2}+b\left\|\nabla u_k\right\|_{2}^{2}\left(\|\nabla u_k^{\pm}\|_{2}^{2}-\frac{1}{2}K_V(u_k)\right)-\frac{a}{2}K_{V}(u_k)\\=&\int_{V}\left|u_k^{\pm}\right|^{p} \log \left(u_k^{\pm}\right)^{2}\,d\mu\\\leq &\int_V (\varepsilon|u_k^{\pm}|^{2}+C_\varepsilon|u_k^{\pm}|^{q})\,d\mu\\\leq &\varepsilon\|u^{\pm}\|^2_{\mathcal{H}_\lambda}+C_\varepsilon\|u^{\pm}\|^q_{\mathcal{H}_\lambda},
\end{aligned}
$$
which implies that
$\|u_k^{\pm}\|_{\mathcal{H}_\lambda}\geq C>0.$ Then we get that $u_\lambda^{\pm}\neq 0$.
Therefore, for $p>6$, we have
$$
\begin{aligned}
m_{\lambda}=&\lim _{k \rightarrow\infty}\left[(\frac{1}{2}-\frac{1}{p})\|u_k\|_{\mathcal{H}_\lambda}^{2}+(\frac{1}{4}-\frac{1}{p})b\left\|\nabla u_k\right\|_{2}^{4}+\frac{2}{p^2}\int_V |u_k|^p\,d\mu\right]\\ \geq&
\lim _{k \rightarrow\infty}\frac{2}{p^2}\int_V |u_k|^p\,d\mu\\=&
\frac{2}{p^2}\int_V |u_\lambda|^p\,d\mu\\>& 0.
\end{aligned}
$$
By (\ref{3.8}), weak lower semicontinuity of norm, Fatou's lemma and Lebesgue dominated theorem, we get that
$$
\begin{aligned}
&\|u_\lambda^{\pm}\|_{\mathcal{H}_\lambda}^{2}+b\left\|\nabla u_\lambda\right\|_{2}^{2}\left(\|\nabla u_\lambda^{\pm}\|_{2}^{2}-\frac{1}{2}K_V(u_\lambda)\right)-\frac{a}{2}K_{V}(u_\lambda)-\int_{V}\left(|u_\lambda^{\pm}|^{p} \log \left(u_\lambda^{\pm}\right)^{2}\right)^-\,d\mu \\
\leq & \liminf _{k \rightarrow\infty}\left[\|u_k^{\pm}\|_{\mathcal{H}_\lambda}^{2}+b\left\|\nabla u_k\right\|_{2}^{2}\left(\|\nabla u_k^{\pm}\|_{2}^{2}-\frac{1}{2}K_V(u_k)\right)-\frac{a}{2}K_{V}(u_k)-\int_{V}\left(|u_k^{\pm}|^{p} \log \left(u_k^{\pm}\right)^{2}\right)^-\,d\mu\right] \\
= & \liminf _{k \rightarrow\infty} \int_{V}\left(|u_k^{\pm}|^{p} \log \left(u_k^{\pm}\right)^{2}\right)^+\,d\mu \\
= & \int_{V}\left(|u_\lambda^{\pm}|^{p} \log \left(u_\lambda^{\pm}\right)^{2}\right)^+\,d\mu,
\end{aligned}
$$
which means that
\begin{equation*}
(J_{\lambda}^{\prime}\left(u_{\lambda}\right), u_{\lambda}^{\pm})=\|u_\lambda^{\pm}\|_{\mathcal{H}_\lambda}^{2}+b\left\|\nabla u_\lambda\right\|_{2}^{2}\left(\|\nabla u_\lambda^{\pm}\|_{2}^{2}-\frac{1}{2}K_V(u_\lambda)\right)-\frac{a}{2}K_{V}(u_\lambda)-\int_{V}|u_\lambda^{\pm}|^{p} \log \left(u_\lambda^{\pm}\right)^{2}\,d\mu\leq0.
\end{equation*}
By Lemma \ref{l5}, there exist two constants $s, t \in(0,1]$ such that $\tilde{u}=s u_{\lambda}^{+}+t u_{\lambda}^{-} \in \mathcal{M}_{\lambda}$. Then
$$
\begin{aligned}
m_{\lambda} \leq &J_{\lambda}(\widetilde{u})=J_{\lambda}\left(\widetilde{u}\right)-\frac{1}{p} (J_{\lambda}^{\prime}\left(\widetilde{u}\right), \widetilde{u})\\=&(\frac{1}{2}-\frac{1}{p})\|\widetilde{u}\|_{\mathcal{H}_\lambda}^{2}+b(\frac{1}{4}-\frac{1}{p})\left\|\nabla \widetilde{u}\right\|_{2}^{4}+\frac{2}{p^2}\int_V |\widetilde{u}|^p\,d\mu\\=& (\frac{1}{2}-\frac{1}{p})\|s u_{\lambda}^{+}+t u_{\lambda}^{-}\|_{\mathcal{H}_\lambda}^{2}+b(\frac{1}{4}-\frac{1}{p})\left\|\nabla s u_{\lambda}^{+}+t u_{\lambda}^{-}\right\|_{2}^{4}+\frac{2}{p^2}\int_V |s u_{\lambda}^{+}+t u_{\lambda}^{-}|^p\,d\mu\\ \leq&(\frac{1}{2}-\frac{1}{p})\|u_{\lambda}\|_{\mathcal{H}_\lambda}^{2}+b(\frac{1}{4}-\frac{1}{p})\left\|\nabla u_{\lambda}\right\|_{2}^{4}+\frac{2}{p^2}\int_V |u_{\lambda}|^p\,d\mu\\ \leq &
\liminf_{k \rightarrow\infty}\left[(\frac{1}{2}-\frac{1}{p})\|u_k\|_{\mathcal{H}_\lambda}^{2}+(\frac{1}{4}-\frac{1}{p})b\left\|\nabla u_k\right\|_{2}^{4}+\frac{2}{p^2}\int_V |u_k|^p\,d\mu\right]\\ =&
\liminf_{k \rightarrow\infty}\left[J_{\lambda}\left(u_{k}\right)-\frac{1}{p} (J_{\lambda}^{\prime}\left(u_{k}\right), u_{k})\right] \\=& m_\lambda.
\end{aligned}
$$
This implies that $s=t=1$. Hence $u_{\lambda} \in \mathcal{M}_{\lambda}$ and  $J_{\lambda}\left(u_{\lambda}\right)=m_{\lambda}>0.$

\end{proof}

The following lemma completes the proof of Theorem \ref{t1}.

\begin{lm}\label{l8}
If $u \in \mathcal{M}_{\lambda}$ with $J_{\lambda}(u)=m_{\lambda}$, then $u$ is a sign-changing solution of~the equation (\ref{0.2}). Moreover, $m_{\lambda}>2 c_{\lambda}$.
\end{lm}

\begin{proof}
We assume by contradiction that $u \in \mathcal{M}_{\lambda}$ with $J_{\lambda}(u)=m_{\lambda}$, but $u$ is not a solution of~the equation (\ref{0.2}). Then we can find a function $\phi \in C_{c}(V)$ such that

$$
\int_{V}(a \nabla u \nabla \phi+h_\lambda(x) u \phi)\,d \mu+b \int_{V}|\nabla u|^{2}\,d \mu \int_{V} \nabla u \nabla \phi\,d \mu-\int_{V}|u|^{p-2}u\phi\log u^2\,d\mu\leq -1.
$$
which implies that, for some $\varepsilon>0$ small enough,
$$
\left(J_{\lambda}^{\prime}\left(s u^{+}+t u^{-}+\sigma \phi\right), \phi\right)\leq-\frac{1}{2},\quad |s-1|+|t-1|+|\sigma| \leq \varepsilon.
$$

We consider $J_{\lambda}\left(s u^{+}+t u^{-}+\varepsilon \eta(s, t) \phi\right)$, where $\eta$ is a cutoff function such that
$$
\eta(s, t)= 
\begin{cases}1, & \text { if }|s-1| \leq \frac{1}{2} \varepsilon \text { and }|t-1| \leq \frac{1}{2} \varepsilon, \\ 0, & \text { if }|s-1| \geq \varepsilon \text { or }|t-1| \geq \varepsilon.
\end{cases}
$$
If $|s-1| \leq \varepsilon$ and $|t-1| \leq \varepsilon$, we have
$$
\begin{aligned}
J_{\lambda}\left(s u^{+}+t u^{-}+\varepsilon \eta(s, t) \phi\right) &  =J_{\lambda}\left(s u^{+}+t u^{-}\right)+\int_{0}^{1} \left(J_{\lambda}^{\prime}\left(s u^{+}+t u^{-}+\sigma \varepsilon \eta(s, t) \phi\right),\varepsilon \eta(s, t) \phi\right) d \sigma \\
& \leq J_{\lambda}\left(s u^{+}+t u^{-}\right)-\frac{1}{2} \varepsilon \eta(s, t).
\end{aligned}
$$
If $|s-1| \geq \varepsilon$ or $|t-1| \geq \varepsilon,\,  \eta(s, t)=0$, the above estimate is obvious. Since $u \in \mathcal{M}_{\lambda}$, for $(s, t) \neq(1,1)$, by Lemma \ref{l3}, we have 
$$
J_{\lambda}\left(s u^{+}+t u^{-}+\varepsilon \eta(s, t) \phi\right) \leq J_{\lambda}\left(s u^{+}+t u^{-}\right)<J_{\lambda}(u), \quad (s, t) \neq(1,1).
$$
For $(s, t)=(1,1)$,
$$
J_{\lambda}\left(s u^{+}+t u^{-}+\varepsilon \eta(s, t) \phi\right) \leq J_{\lambda}\left(s u^{+}+t u^{-}\right)-\frac{1}{2} \varepsilon \eta(1,1)=J_{\lambda}(u)-\frac{1}{2} \varepsilon.
$$
In any case, we have $J_{\lambda}\left(s u^{+}+t u^{-}+\varepsilon \eta(s, t) \phi\right)<J_{\lambda}(u)=m_{\lambda}$. In particular, for $0<\delta<1-\varepsilon$,
$$
\sup _{\delta \leq s, t \leq 2-\delta} J_{\lambda}\left(s u^{+}+t u^{-}+\varepsilon \eta(s, t) \phi\right)=\tilde{m}_{\lambda}<m_{\lambda}.
$$
Let $v=s u^{+}+t u^{-}+\varepsilon \eta(s, t) \phi$. Define
$$
F_{1}(s, t)=\left(J_{\lambda}^{\prime}(v), v^{+}\right),\qquad F_{2}(s, t)= \left(J_{\lambda}^{\prime}(v), v^{-}\right).
$$
Since $u\in\mathcal{M}_\lambda$, we have
\begin{equation}\label{3.7}
\begin{aligned}
\int_{V}\left|u^{+}\right|^{p} \log \left(u^{+}\right)^{2}\,d\mu= &\|u^{+}\|_{\mathcal{H}_\lambda}^{2}+b\left\|\nabla u^{+}\right\|_{2}^{4}-\frac{a}{2}K_{V}(u)+\frac{b}{2}K^2_V(u) 
+b\left\|\nabla u^{+}\right\|_{2}^{2}\left\|\nabla u^{-}\right\|_{2}^{2}\\&-\frac{b}{2}K_V(u)\left(\left\|\nabla u^{+}\right\|_{2}^{2}+\left\|\nabla u^{-}\right\|_{2}^{2}\right)
-bK_V(u)\|\nabla u^{+}\|^2_2,
\end{aligned}
\end{equation}
and 
\begin{equation}\label{4.7}
\begin{aligned}
\int_{V}\left|u^{-}\right|^{p} \log \left(u^{-}\right)^{2}\,d\mu= &\|u^{-}\|_{\mathcal{H}_\lambda}^{2}+b\left\|\nabla u^{-}\right\|_{2}^{4}-\frac{a}{2}K_{V}(u)+\frac{b}{2}K^2_V(u) 
+b\left\|\nabla u^{+}\right\|_{2}^{2}\left\|\nabla u^{-}\right\|_{2}^{2}\\&-\frac{b}{2}K_V(u)\left(\left\|\nabla u^{+}\right\|_{2}^{2}+\left\|\nabla u^{-}\right\|_{2}^{2}\right)
-bK_V(u)\|\nabla u^{-}\|^2_2.
\end{aligned}
\end{equation}
By the definition of $\eta$, for $s=\delta<1-\varepsilon$ and $t \in(\delta, 2-\delta)$, we have $\eta(s, t)=0$ and $s<t$. Hence by (\ref{3.7}) and $p>6$, we get that
$$
\begin{aligned}
&F_{1}(\delta, t)=(J_{\lambda}^{\prime}(\delta u^++tu^-),\delta u^{+})\\=&\delta^{2}\|u^{+}\|_{\mathcal{H}_\lambda}^{2}+b \delta^{4}\left\|\nabla u^{+}\right\|_{2}^{4}-\delta^p\int_{V}\left|u^{+}\right|^{p} \log \left(u^{+}\right)^{2}\,d\mu-\delta^p\log \delta^2\int_{V}\left|u^{+}\right|^{p}\,d\mu-\frac{a}{2}\delta tK_{V}(u)\\&+\frac{b}{2}\delta^2t^2K^2_V(u) 
+b \delta^{2} t^{2}\left\|\nabla u^{+}\right\|_{2}^{2}\left\|\nabla u^{-}\right\|_{2}^{2}-\frac{b}{2}\delta tK_V(u)\left(\delta^2\left\|\nabla u^{+}\right\|_{2}^{2}+t^2\left\|\nabla u^{-}\right\|_{2}^{2}\right)
\\&-b\delta^3tK_V(u)\|\nabla u^{+}\|^2_2\\ \geq &\delta^{2}\|u^{+}\|_{\mathcal{H}_\lambda}^{2}+b \delta^{4}\left\|\nabla u^{+}\right\|_{2}^{4}-\delta^p\int_{V}\left|u^{+}\right|^{p} \log \left(u^{+}\right)^{2}\,d\mu-\delta^p\log \delta^2\int_{V}\left|u^{+}\right|^{p}\,d\mu-\frac{a}{2}\delta^2K_{V}(u)\\&+\frac{b}{2}\delta^4K^2_V(u) 
+b \delta^{4}\left\|\nabla u^{+}\right\|_{2}^{2}\left\|\nabla u^{-}\right\|_{2}^{2}-\frac{b}{2}\delta^4K_V(u)\left(\left\|\nabla u^{+}\right\|_{2}^{2}+\left\|\nabla u^{-}\right\|_{2}^{2}\right)
-b\delta^4K_V(u)\|\nabla u^{+}\|^2_2\\=&(\delta^2-\delta^p)\|u^{+}\|_{\mathcal{H}_\lambda}^{2}+b(\delta^{4}-\delta^p)\left\|\nabla u^{+}\right\|_{2}^{4}-\delta^p\log \delta^2\int_{V}\left|u^{+}\right|^{p}\,d\mu-\frac{a}{2}(\delta^2-\delta^p)K_{V}(u)\\&+\frac{b}{2}(\delta^4-\delta^p)K^2_V(u) 
+b (\delta^{4}-\delta^p)\left\|\nabla u^{+}\right\|_{2}^{2}\left\|\nabla u^{-}\right\|_{2}^{2}-\frac{b}{2}(\delta^4-\delta^p)K_V(u)\left(\left\|\nabla u^{+}\right\|_{2}^{2}+\left\|\nabla u^{-}\right\|_{2}^{2}\right)
\\&-b(\delta^4-\delta^p)K_V(u)\|\nabla u^{+}\|^2_2\\>&-\delta^p\log \delta^2\int_{V}\left|u^{+}\right|^{p}\,d\mu\\>&0.
\end{aligned}
$$
For $s=2-\delta>1+\varepsilon$ and $t \in(\delta, 2-\delta)$, we have $\eta(s, t)=0$ and $s>t$. Similarly,
$$
\begin{aligned}
&F_{1}(2-\delta, t)=(J_{\lambda}^{\prime}((2-\delta)u^++tu^-),(2-\delta)u^{+})\\=&(2-\delta)^{2}\|u^{+}\|_{\mathcal{H}_\lambda}^{2}+b (2-\delta)^{4}\left\|\nabla u^{+}\right\|_{2}^{4}-(2-\delta)^p\int_{V}\left|u^{+}\right|^{p} \log \left(u^{+}\right)^{2}\,d\mu\\&-(2-\delta)^p\log (2-\delta)^2\int_{V}\left|u^{+}\right|^{p}\,d\mu-\frac{a}{2}(2-\delta)tK_{V}(u)+\frac{b}{2}(2-\delta)^2t^2K^2_V(u) 
\\&+b (2-\delta)^{2} t^{2}\left\|\nabla u^{+}\right\|_{2}^{2}\left\|\nabla u^{-}\right\|_{2}^{2}-\frac{b}{2}(2-\delta)tK_V(u)\left((2-\delta)^2\left\|\nabla u^{+}\right\|_{2}^{2}+t^2\left\|\nabla u^{-}\right\|_{2}^{2}\right)
\\&-b(2-\delta)^3tK_V(u)\|\nabla u^{+}\|^2_2\\ \leq &(2-\delta)^{2}\|u^{+}\|_{\mathcal{H}_\lambda}^{2}+b (2-\delta)^{4}\left\|\nabla u^{+}\right\|_{2}^{4}-(2-\delta)^p\int_{V}\left|u^{+}\right|^{p} \log \left(u^{+}\right)^{2}\,d\mu\\&-(2-\delta)^p\log (2-\delta)^2\int_{V}\left|u^{+}\right|^{p}\,d\mu-\frac{a}{2}(2-\delta)^2K_{V}(u)+\frac{b}{2}(2-\delta)^4K^2_V(u) 
\\&+b (2-\delta)^{4}\left\|\nabla u^{+}\right\|_{2}^{2}\left\|\nabla u^{-}\right\|_{2}^{2}-\frac{b}{2}(2-\delta)^4K_V(u)\left(\left\|\nabla u^{+}\right\|_{2}^{2}+\left\|\nabla u^{-}\right\|_{2}^{2}\right)
-b(2-\delta)^4K_V(u)\|\nabla u^{+}\|^2_2\\=&\left((2-\delta)^2-(2-\delta)^p\right)\|u^{+}\|_{\mathcal{H}_\lambda}^{2}+b\left((2-\delta)^{4}-(2-\delta)^p\right)\left\|\nabla u^{+}\right\|_{2}^{4}-(2-\delta)^p\log (2-\delta)^2\int_{V}\left|u^{+}\right|^{p}\,d\mu\\&-\frac{a}{2}\left((2-\delta)^2-(2-\delta)^p\right)K_{V}(u)+\frac{b}{2}\left((2-\delta)^2-(2-\delta)^p\right)K^2_V(u) 
\\&+b\left((2-\delta)^{4}-(2-\delta)^p\right)\left\|\nabla u^{+}\right\|_{2}^{2}\left\|\nabla u^{-}\right\|_{2}^{2}-\frac{b}{2}\left((2-\delta)^4-(2-\delta)^p\right)K_V(u)\left(\left\|\nabla u^{+}\right\|_{2}^{2}+\left\|\nabla u^{-}\right\|_{2}^{2}\right)
\\&-b\left((2-\delta)^4-(2-\delta)^p\right)K_V(u)\|\nabla u^{+}\|^2_2\\<&-(2-\delta)^p\log (2-\delta)^2\int_{V}\left|u^{+}\right|^{p}\,d\mu\\<& 0.
\end{aligned}
$$
Hence, we have
$$
F_{1}(\delta, t)>0,\qquad F_{1}(2-\delta, t)<0,\quad t \in(\delta, 2-\delta).
$$
By (\ref{4.7}) and similar arguments as above, we have
$$
F_{2}(s, \delta)>0,\qquad F_{2}(s, 2-\delta)<0,\quad s \in(\delta, 2-\delta).
$$
By the Miranda's theorem \cite{K}, there exists $\left(s_{0}, t_{0}\right) \in(\delta, 2-\delta) \times(\delta, 2-\delta)$ such that $$\widetilde{u}=s_{0} u^{+}+t_{0} u^{-}+\varepsilon \eta\left(s_{0}, t_{0}\right) \phi \in\mathcal{M}_{\lambda},$$ and $J_{\lambda}(\widetilde{u})<m_{\lambda},$ which contradicts the definition of $m_{\lambda}$.

Finally, we claim that $m_{\lambda}>2 c_{\lambda}$. In fact, let $u \in \mathcal{M}_{\lambda}$ such that $J_{\lambda}(u)=m_{\lambda}$, then $u^{ \pm} \neq 0$. By similar arguments to the proof of Lemma \ref{4} and Lemma \ref{l5}, one can get that there exists a unique $s_{u^{+}} \in(0,1]$ such that $s_{u^{+}} u^{+} \in \mathcal{N}_{\lambda}$, and a unique $t_{u^{-}} \in(0,1]$ such that $t_{u^{-}} u^{-} \in \mathcal{N}_{\lambda}$. Moreover, we can obtain that $c_{\lambda}>0$ is achieved and that if $u \in \mathcal{N}_\lambda$ with $J_{\lambda}(u)=c_{\lambda}$, then $u$ is a least energy solution.

By the definition of $J_{\lambda}$ and $K_{V}(u)<0$, we have
$$
\begin{aligned}
&J_{\lambda}(s_{u^{+}}u^++t_{u^{-}}u^-)\\=&J_{\lambda}\left(s_{u^{+}}u^{+}\right)+J_{\lambda}\left(t_{u^{-}}u^{-}\right)-\frac{a}{2}s_{u^{+}}t_{u^{-}}K_{V}(u)+\frac{b}{4}s_{u^{+}}^2t_{u^{-}}^2K^2_V(u)+\frac{b}{2}\left\|\nabla (s_{u^{+}}u^{+})\right\|_{2}^{2}\left\|\nabla (t_{u^{-}}u^{-})\right\|_{2}^{2}\\ &-\frac{b}{2}s_{u^{+}}t_{u^{-}}K_V(u)\left(\left\|\nabla (s_{u^{+}}u^{+})\right\|_{2}^{2}+\left\|\nabla (t_{u^{-}}u^{-})\right\|_{2}^{2}\right)\\>&J_{\lambda}\left(s_{u^{+}}u^{+}\right)+J_{\lambda}\left(t_{u^{-}}u^{-}\right).
\end{aligned}
$$
By Lemma \ref{l3}, we get that
$$
m_{\lambda}=J_{\lambda}\left(u\right) \geq J_{\lambda}\left(s_{u^{+}} u^{+}+t_{u^{-}} u^{-}\right)>J_{\lambda}\left(s_{u^{+}} u^{+}\right)+J_{\lambda}\left(t_{u^{-}} u^{-}\right) \geq 2 c_{\lambda}.
$$
We complete the proof.

\end{proof}

{\bf Proof of Theorem \ref{t2}} Similar to the proof of Theorem \ref{t1}, we can also obtain the existence of a least energy sign-changing solution $u_{0}$ of the equation (\ref{1.8}), which achieves the minimum $m_{\Omega}$ of the functional $J_{\Omega}$ in $\mathcal{M}_{\Omega}$ and the least energy solution of the equation (\ref{1.8}), which achieves the minimum $c_{\Omega}$ of the functional $J_{\Omega}$ in $\mathcal{N}_{\Omega}$. Moreover, we can also get that $m_{\Omega}>2 c_{\Omega}$.\qed

\section{Convergence of least energy sign-changing solutions}
In this section, we prove Theorem \ref{t3}. In order to prove the asymptotic behavior of the least energy sign-changing $u_{\lambda} \in \mathcal{H}_{\lambda}$ of the equation (\ref{0.2}) as $\lambda \rightarrow\infty$, we first show some important lemmas.

\begin{lm}\label{s1}
For any $u \in \mathcal{M}_{\lambda}$, there exists $\sigma>0$ such that $\|u\|_{\mathcal{H}_{\lambda}} \geq\|u\|_{H^1} \geq \sigma$, where $\sigma$ does not depend on $\lambda.$
\end{lm}

\begin{proof}
Let $\varepsilon=\frac{1}{2}$ in (\ref{3.9}), then for $t\neq 0$, we get that
\begin{equation}\label{4.1}
|t|^p|\log t^2|\leq (\frac{1}{2} |t|^{2}+C|t|^{q}),\quad q>p>6.
\end{equation}
Since $u \in \mathcal{M}_{\lambda}$ and $K_V(u)<0$, by Lemma \ref{l2} and (\ref{4.1}), we have that
$$
\begin{aligned}
0=&(J_{\lambda}^{\prime}(u),u^{\pm})\\=&\|u^{\pm}\|_{\mathcal{H}_\lambda}^{2}+b\left\|\nabla u^{\pm}\right\|_{2}^{4}-\int_{V}\left|u^{\pm}\right|^{p} \log \left(u^{\pm}\right)^{2}\,d\mu-\frac{a}{2}K_{V}(u)+\frac{b}{2}K^2_V(u) 
\\&+b\left\|\nabla u^{+}\right\|_{2}^{2}\left\|\nabla u^{-}\right\|_{2}^{2}-\frac{b}{2}K_V(u)\left(\left\|\nabla u^{+}\right\|_{2}^{2}+\left\|\nabla u^{-}\right\|_{2}^{2}\right)
-bK_V(u)\|\nabla u^{\pm}\|^2_2\\\geq &\|u^{\pm}\|_{\mathcal{H}_\lambda}^{2}+b\left\|\nabla u^{\pm}\right\|_{2}^{4}-\int_{V}\left|u^{\pm}\right|^{p} \log \left(u^{\pm}\right)^{2}\,d\mu\\ \geq & \|u^{\pm}\|_{H^1}^{2}-\frac{1}{2}\|u^{\pm}\|^2_2-C\|u^{\pm}\|^q_q\\ \geq &\|u^{\pm}\|^2_{H^1}-\frac{1}{2}\|u^{\pm}\|^2_2-C\|u^{\pm}\|^q_2\\ \geq &\|u^{\pm}\|^2_{H^1}-\frac{1}{2}\|u^{\pm}\|^2_{H^1}-C\|u^{\pm}\|^q_{H^1}\\ =& \frac{1}{2}\|u^{\pm}\|^2_{H^1}-C\|u^{\pm}\|^q_{H^1}.
\end{aligned}
$$
This implies that
$$\|u^{\pm}\|_{\mathcal{H}_\lambda}\geq \|u^{\pm}\|_{H^1}\geq \left(\frac{1}{C}\right)^{\frac{1}{q-2}}>0.$$

As a consequence, we get that
$$
\|u\|_{\mathcal{H}_{\lambda}}^{2} \geq\|u\|_{H^1}^{2}=\left\|u^{+}\right\|_{H^1}^{2}+\left\|u^{-}\right\|_{H^1}^{2}-K_{V}(u)>\left\|u^{+}\right\|_{H^1}^{2}+\left\|u^{-}\right\|_{H^1}^{2} \geq 2\left(\frac{1}{C}\right)^{\frac{1}{q-2}}>0.
$$
By choosing $\sigma=2\left(\frac{1}{C}\right)^{\frac{1}{q-2}}>0$, we get the desired result.

\end{proof}

\begin{lm}\label{s2}
Let $\left\{u_{k}\right\} \subset \mathcal{M}_{\lambda}$ satisfy $\lim \limits_{k \rightarrow \infty} J_{\lambda}\left(u_{k}\right)=m_{\lambda}$. Then there exists $c_{0}>0$ such that $\left\|u_{k}\right\|_{\mathcal{H}_{\lambda}} \leq c_{0}$, where $c_0$ does not depend on $\lambda.$
\end{lm}

\begin{proof}
Note that $\mathcal{M}_{\Omega} \subset \mathcal{M}_{\lambda}$, then we have $0<m_{\lambda} \leq m_{\Omega}$. Since $\{u_k\}\subset\mathcal{M}_\lambda$, by (\ref{0.3}), we get that $$(J'_\lambda(u_k),u_k)=(J'_\lambda(u_k),u_k^+)+(J'_\lambda(u_k),u_k^-) =0.$$
A direct calculation yields that
$$
\begin{aligned}
\lim _{k \rightarrow\infty} J_{\lambda}\left(u_{k}\right) & =\lim _{k \rightarrow\infty}\left[J_{\lambda}\left(u_{k}\right)-\frac{1}{p} (J_{\lambda}^{\prime}\left(u_{k}\right), u_{k})\right] \\
& =\lim _{k \rightarrow\infty}\left[(\frac{1}{2}-\frac{1}{p})\|u_k\|_{\mathcal{H}_\lambda}^{2}+(\frac{1}{4}-\frac{1}{p})b\left\|\nabla u_k\right\|_{2}^{4}+\frac{2}{p^2}\int_V |u_k|^p\,d\mu\right] \\
& =m_{\lambda}.
\end{aligned}
$$
Then we have
$$\lim\limits_{k\rightarrow\infty}\|u_k\|_{\mathcal{H}_\lambda}^{2}\leq Cm_\lambda\leq Cm_\Omega.$$
This implies that there exists $c_0>0$, which is independent of $\lambda$, such that
$\|u_k\|_{\mathcal{H}_\lambda}\leq c_0.$

\end{proof}

Now we prove a crucial lemma for the relation
between $m_{\lambda}$ and $m_{\Omega}$ as $\lambda\rightarrow\infty.$

\begin{lm}\label{s3}
$m_{\lambda} \rightarrow m_{\Omega}$ as $\lambda \rightarrow\infty$.
\end{lm}

\begin{proof}
For any $\lambda>0$, since $m_{\lambda} \leq m_{\Omega}$ , up to a subsequence, we may take a sequence $\lambda_{k} \rightarrow\infty$ such that
\begin{equation}\label{4.2}
\lim\limits_{k \rightarrow \infty} m_{\lambda_{k}}=\eta \leq m_{\Omega}.
\end{equation}
By theorem \ref{t1}, there exists a sequence $\{u_{\lambda_{k}}\}\subset \mathcal{M}_{\lambda_{k}}$, least energy sign-changing solutions of the equation (\ref{0.2}), such that $J_{\lambda_{k}}(u_{\lambda_k})=m_{\lambda_{k}}>0$. By Lemma \ref{s2}, one sees that $\left\{u_{\lambda_{k}}\right\}$ is uniformly bounded in $\mathcal{H}_{\lambda_{k}}$. Consequently, $\left\{u_{\lambda_{k}}\right\}$ is also bounded in $H^1(V)$ and thus, up to a subsequence, there exists some $u_{0} \in H^1(V)$ such that
\[
\left\{\begin{array}{l}
u_{\lambda_{k}} \rightharpoonup u_{0},\quad\text { weakly in } H^1(V),\\
u_{\lambda_{k}} \rightarrow u_{0}, \quad\text { pointwise in } V, \\
u_{\lambda_{k}} \rightarrow u_{0},\quad \text { strongly in } \ell^{q}(V), q\in[2,\infty].
\end{array}\right.
\]
Then it follows from Lemma \ref{s1} that $u_{0} \not \equiv 0$. We claim that $u_0|_{\Omega^{c}}=0$. In fact, if there exists a vertex $x_{0} \in \Omega^{c}$ such that $u_{0}\left(x_{0}\right) \neq 0$. Since $u_k:=u_{\lambda_{k}} \in \mathcal{M}_{\lambda_{k}}$, by (\ref{0.3}), we get that $$(J'_{\lambda_k}(u_k),u_k)=(J'_{\lambda_k}(u_k),u_k^+)+(J'_{\lambda_k}(u_k),u_k^-) =0.$$
Moreover, since $h\left(x_{0}\right)>0, u_{k}(x_0) \rightarrow u_{0}(x_0) \neq 0$ and $\lambda_{k} \rightarrow\infty$, we get that
$$
\begin{aligned}
J_{\lambda_k}\left(u_{k}\right) =&\left[J_{\lambda_k}\left(u_{k}\right)-\frac{1}{p} (J_{\lambda_k}^{\prime}\left(u_{k}\right), u_{k})\right] \\
=&\left[(\frac{1}{2}-\frac{1}{p})\|u_k\|_{\mathcal{H}_\lambda}^{2}+(\frac{1}{4}-\frac{1}{p})b\left\|\nabla u_k\right\|_{2}^{4}+\frac{2}{p^2}\int_V |u_k|^p\,d\mu\right] \\
\geq &(\frac{1}{2}-\frac{1}{p})\|u_k\|_{\mathcal{H}_\lambda}^{2}\\ \geq & (\frac{1}{2}-\frac{1}{p}) \lambda_k h(x_0)|u_k(x_0)|^2\\ \rightarrow &\infty,\quad k\rightarrow \infty,
\end{aligned}
$$
which contradicts (\ref{4.2}). Hence the claim is completed.

Since $u_0|_{\Omega^{c}}=0$, by weak lower semicontinuity of the norm $\|\cdot\|_{H^1}$, Fatou's lemma and Lebesgue dominated theorem, we get that
$$
\begin{aligned}
&\|u_0^{+}\|_{\mathcal{H}_{0}^{1}(\Omega)}^{2}+b\left\|\nabla u_0\right\|_{\ell^2(\Omega\cup\partial\Omega)}^{2}\left(\|\nabla u_0^{+}\|_{\ell^2(\Omega\cup\partial\Omega)}^{2}-\frac{1}{2}K_\Omega(u_0)\right)-\frac{a}{2}K_{\Omega}(u_0)-\int_{\Omega}\left(u_0^{+}|^{p} \log \left(u_0^{+}\right)^{2}\right)^-\,d\mu\\ \leq & \|u_0^{+}\|_{H^1}^{2}+b\left\|\nabla u_0\right\|_{2}^{2}\left(\|\nabla u_0^{+}\|_{2}^{2}-\frac{1}{2}K_V(u_0)\right)-\frac{a}{2}K_{V}(u_0)-\int_{V}\left(u_0^{+}|^{p} \log \left(u_0^{+}\right)^{2}\right)^-\,d\mu \\
\leq & \liminf _{k \rightarrow\infty}\left[\|u_k^{+}\|_{\mathcal{H}_\lambda}^{2}+b\left\|\nabla u_k\right\|_{2}^{2}\left(\|\nabla u_k^{+}\|_{2}^{2}-\frac{1}{2}K_V(u_k)\right)-\frac{a}{2}K_{V}(u_k)-\int_{V}\left(|u_k^{+}|^{p} \log \left(u_k^{+}\right)^{2}\right)^-\,d\mu\right] \\
= & \liminf _{k \rightarrow\infty} \int_{V}\left(|u_k^{+}|^{p} \log \left(u_k^{+}\right)^{2}\right)^+\,d\mu \\
= & \int_{V}\left(|u_0^{+}|^{p} \log \left(u_0^{+}\right)^{2}\right)^+\,d\mu\\
= & \int_{\Omega}\left(|u_0^{+}|^{p} \log \left(u_0^{+}\right)^{2}\right)^+\,d\mu.
\end{aligned}
$$
The previous inequality means that
$$
\begin{aligned}
&(J_{\Omega}^{\prime}\left(u_{0}\right), u_{0}^{+})\\=&\|u_0^{+}\|_{\mathcal{H}_{0}^{1}(\Omega)}^{2}+b\left\|\nabla u_0\right\|_{\ell^2(\Omega\cup\partial\Omega)}^{2}\left(\|\nabla u_0^{+}\|_{\ell^2(\Omega\cup\partial\Omega)}^{2}-\frac{1}{2}K_\Omega(u_0)\right)-\frac{a}{2}K_{\Omega}(u_0)-\int_{\Omega}|u_0^{+}|^{p} \log \left(u_0^{+}\right)^{2}\,d\mu\leq0.
\end{aligned}
$$
Similarly, it holds that
$$
\begin{aligned}
&(J_{\Omega}^{\prime}\left(u_{0}\right), u_{0}^{-})\\=&\|u_0^{-}\|_{\mathcal{H}_{0}^{1}(\Omega)}^{2}+b\left\|\nabla u_0\right\|_{\ell^2(\Omega\cup\partial\Omega)}^{2}\left(\|\nabla u_0^{-}\|_{\ell^2(\Omega\cup\partial\Omega)}^{2}-\frac{1}{2}K_\Omega(u_0)\right)-\frac{a}{2}K_{\Omega}(u_0)-\int_{\Omega}|u_0^{-}|^{p} \log \left(u_0^{-}\right)^{2}\,d\mu\leq0.
\end{aligned}
$$
By Lemma \ref{l7} and Lemma \ref{l6}, there exist two constants $s, t \in(0,1]$ such that $\tilde{u}_0=s u_{0}^{+}+t u_{0}^{-} \in \mathcal{M}_{\Omega}$. Then
$$
\begin{aligned}
m_{\Omega} \leq &J_{\Omega}(\widetilde{u}_0)=J_{\Omega}\left(\widetilde{u}_0\right)-\frac{1}{p} (J_{\Omega}^{\prime}\left(\widetilde{u}_0\right), \widetilde{u}_0)\\=&(\frac{1}{2}-\frac{1}{p})\|\widetilde{u}_0\|_{\mathcal{H}_{0}^{1}(\Omega)}^{2}+b(\frac{1}{4}-\frac{1}{p})\left\|\nabla \widetilde{u}_0\right\|_{\ell^2(\Omega\cup\partial\Omega)}^{4}+\frac{2}{p^2}\int_\Omega |\widetilde{u}_0|^p\,d\mu\\=& (\frac{1}{2}-\frac{1}{p})\|s u_{0}^{+}+t u_{0}^{-}\|_{\mathcal{H}_{0}^{1}(\Omega)}^{2}+b(\frac{1}{4}-\frac{1}{p})\left\|\nabla s u_{0}^{+}+t u_{0}^{-}\right\|_{\ell^2(\Omega\cup\partial\Omega)}^{4}+\frac{2}{p^2}\int_\Omega |s u_{0}^{+}+t u_{0}^{-}|^p\,d\mu\\ \leq&(\frac{1}{2}-\frac{1}{p})\|u_{0}\|_{\mathcal{H}_\lambda}^{2}+b(\frac{1}{4}-\frac{1}{p})\left\|\nabla u_{0}\right\|_{2}^{4}+\frac{2}{p^2}\int_V |u_{0}|^p\,d\mu\\ \leq &
\liminf_{k \rightarrow\infty}\left[(\frac{1}{2}-\frac{1}{p})\|u_k\|_{\mathcal{H}_\lambda}^{2}+b(\frac{1}{4}-\frac{1}{p})\left\|\nabla u_k\right\|_{2}^{4}+\frac{2}{p^2}\int_V |u_k|^p\,d\mu\right]\\ =&
\liminf_{k \rightarrow\infty}\left[J_{\lambda_k}\left(u_{k}\right)-\frac{1}{p} (J_{\lambda_k}^{\prime}\left(u_{k}\right), u_{k})\right] \\=&\liminf _{k \rightarrow\infty} J_{\lambda_{k}}\left(u_{\lambda_{k}}\right)=\eta \leq m_{\Omega}.
\end{aligned}
$$
Hence we get that
$$\lim_{\lambda\rightarrow\infty} m_{\lambda}= m_{\Omega}.$$

\end{proof}

{\bf Proof of Theorem \ref{t3}} Let $\{u_{k}\} \subset \mathcal{M}_{\lambda_{k}}$ satisfy $J_{\lambda_{k}}(u_{k})=m_{\lambda_{k}}$. We prove that $\{u_{k}\}$ converges in $H^1(V)$ to a least energy sign-changing solution $u_{0}$ of the equation (\ref{1.8}).

By Lemma \ref{s2}, we get that $\{u_{k}\} \subset\mathcal{H}_{\lambda_{k}}$ is bounded, and hence bounded in $H^1(V)$. Therefore, up to a subsequence, there exists $u_{0} \in H^1(V)$ such that
\begin{equation}\label{0.8}
\begin{cases}
u_{k} \rightharpoonup u_{0},\quad \text { weakly in } H^1(V),\\
u_{k} \rightarrow u_{0},\quad \text { pointwise in } V, \\
u_{k} \rightarrow u_{0},\quad \text { strongly in } \ell^{q}(V),~ q\in[2,\infty].
\end{cases}
\end{equation}
Then it follows from Lemma \ref{s1} that $u_{0} \not \equiv 0$. As proved in Lemma \ref{s3}, we get that $u_0|_{\Omega^{c}}=0$.

We claim that, as $k \rightarrow\infty$, \begin{equation}\label{0.7}
\lambda_{k} \int_{V} h(x)\left|u_{k}^{ \pm}\right|^{2} d \mu \rightarrow 0,\qquad\text{and}\qquad \int_{V} |\nabla u_{k}^{ \pm}|^2\,d\mu \rightarrow \int_{V} \nabla u_{0}^{ \pm}|^2\,d \mu.
\end{equation}
In fact, if not, we may assume that
$$
\lim _{k \rightarrow\infty} \lambda_{k} \int_{V} h(x)\left|u_{k}^{ \pm}\right|^{2} d \mu=\delta>0.
$$
Since $u_0|_{\Omega^{c}}=0$, by weak lower semicontinuity of the norm $\|\cdot\|_{H^1}$, Fatou's lemma and Lebesgue dominated theorem, we obtain that
$$
\begin{aligned}
&\|u_0^{\pm}\|_{\mathcal{H}_{0}^{1}(\Omega)}^{2}+b\left\|\nabla u_0\right\|_{\ell^2(\Omega\cup\partial\Omega)}^{2}\left(\|\nabla u_0^{\pm}\|_{\ell^2(\Omega\cup\partial\Omega)}^{2}-\frac{1}{2}K_\Omega(u_0)\right)-\frac{a}{2}K_{\Omega}(u_0)-\int_{\Omega}\left(u_0^{\pm}|^{p} \log \left(u_0^{\pm}\right)^{2}\right)^-\,d\mu\\ \leq & \|u_0^{\pm}\|_{H^1}^{2}+b\left\|\nabla u_0\right\|_{2}^{2}\left(\|\nabla u_0^{\pm}\|_{2}^{2}-\frac{1}{2}K_V(u_0)\right)-\frac{a}{2}K_{V}(u_0)-\int_{V}\left(u_0^{\pm}|^{p} \log \left(u_0^{\pm}\right)^{2}\right)^-\,d\mu \\
< &\|u_0^{\pm}\|_{H^1}^{2}+\delta+b\left\|\nabla u_0\right\|_{2}^{2}\left(\|\nabla u_0^{\pm}\|_{2}^{2}-\frac{1}{2}K_V(u_0)\right)-\frac{a}{2}K_{V}(u_0)-\int_{V}\left(u_0^{\pm}|^{p} \log \left(u_0^{\pm}\right)^{2}\right)^-\,d\mu \\
\leq & \liminf _{k \rightarrow\infty}\left[\|u_k^{\pm}\|_{\mathcal{H}_\lambda}^{2}+b\left\|\nabla u_k\right\|_{2}^{2}\left(\|\nabla u_k^{\pm}\|_{2}^{2}-\frac{1}{2}K_V(u_k)\right)-\frac{a}{2}K_{V}(u_k)-\int_{V}\left(|u_k^{\pm}|^{p} \log \left(u_k^{\pm}\right)^{2}\right)^-\,d\mu\right] \\
= & \liminf _{k \rightarrow\infty} \int_{V}\left(|u_k^{\pm}|^{p} \log \left(u_k^{\pm}\right)^{2}\right)^+\,d\mu \\
= & \int_{V}\left(|u_0^{\pm}|^{p} \log \left(u_0^{\pm}\right)^{2}\right)^+\,d\mu\\
= & \int_{\Omega}\left(|u_0^{\pm}|^{p} \log \left(u_0^{\pm}\right)^{2}\right)^+\,d\mu,
\end{aligned}
$$
which means that
\begin{equation}\label{9.0}
\begin{aligned}
&(J_{\Omega}^{\prime}\left(u_{0}\right), u_{0}^{\pm})\\=&\|u_0^{\pm}\|_{\mathcal{H}_{0}^{1}(\Omega)}^{2}+b\left\|\nabla u_0\right\|_{\ell^2(\Omega\cup\partial\Omega)}^{2}\left(\|\nabla u_0^{\pm}\|_{\ell^2(\Omega\cup\partial\Omega)}^{2}-\frac{1}{2}K_\Omega(u_0)\right)-\frac{a}{2}K_{\Omega}(u_0)-\int_{\Omega}|u_0^{\pm}|^{p} \log \left(u_0^{\pm}\right)^{2}\,d\mu<0.
\end{aligned}
\end{equation}

If
$$
\lim _{k \rightarrow\infty} \int_{V} |\nabla u_{k}^{ \pm}|^2\,d \mu>\int_{V} |\nabla u_{0}^{ \pm}|^2\,d \mu.
$$
By similar arguments as above, we can also get (\ref{9.0}). Then it follows from  Lemma \ref{l7} and Lemma \ref{l6} that there exist two constants $s, t \in(0,1)$ such that $\tilde{u}_0=s u_{0}^{+}+t u_{0}^{-} \in \mathcal{M}_{\Omega}$. Therefore, we have that
$$
\begin{aligned}
m_{\Omega} \leq &J_{\Omega}(\widetilde{u}_0)=J_{\Omega}\left(\widetilde{u}_0\right)-\frac{1}{p} (J_{\Omega}^{\prime}\left(\widetilde{u}_0\right), \widetilde{u}_0)\\=&(\frac{1}{2}-\frac{1}{p})\|\widetilde{u}_0\|_{\mathcal{H}_{0}^{1}(\Omega)}^{2}+b(\frac{1}{4}-\frac{1}{p})\left\|\nabla \widetilde{u}_0\right\|_{\ell^2(\Omega\cup\partial\Omega)}^{4}+\frac{2}{p^2}\int_\Omega |\widetilde{u}_0|^p\,d\mu\\=& (\frac{1}{2}-\frac{1}{p})\|s u_{0}^{+}+t u_{0}^{-}\|_{\mathcal{H}_{0}^{1}(\Omega)}^{2}+b(\frac{1}{4}-\frac{1}{p})\left\|\nabla s u_{0}^{+}+t u_{0}^{-}\right\|_{\ell^2(\Omega\cup\partial\Omega)}^{4}+\frac{2}{p^2}\int_\Omega |s u_{0}^{+}+t u_{0}^{-}|^p\,d\mu\\ <&(\frac{1}{2}-\frac{1}{p})\|u_{0}\|_{\mathcal{H}_\lambda}^{2}+b(\frac{1}{4}-\frac{1}{p})\left\|\nabla u_{0}\right\|_{2}^{4}+\frac{2}{p^2}\int_V |u_{0}|^p\,d\mu\\ \leq &
\liminf_{k \rightarrow\infty}\left[(\frac{1}{2}-\frac{1}{p})\|u_k\|_{\mathcal{H}_\lambda}^{2}+(\frac{1}{4}-\frac{1}{p})b\left\|\nabla u_k\right\|_{2}^{4}+\frac{2}{p^2}\int_V |u_k|^p\,d\mu\right]\\ =&
\liminf_{k \rightarrow\infty}\left[J_{\lambda_k}\left(u_{k}\right)-\frac{1}{p} (J_{\lambda_k}^{\prime}\left(u_{k}\right), u_{k})\right] \\=&\liminf _{k \rightarrow\infty} J_{\lambda_{k}}\left(u_{k}\right)\\= &m_{\Omega}.
\end{aligned}
$$
Clearly, this is a contradiction. Hence, we complete the claim (\ref{0.7}).

By (i) of Proposition \ref{o}, (\ref{0.8}), (\ref{0.7}), Lebesgue dominated theorem and $u_0|_{\Omega^{c}}=0$, we derive that
\begin{equation}\label{0.9}
\begin{aligned}
\lim\limits_{k\rightarrow\infty}\|u_k\|_{\mathcal{H}_{\lambda_k}}^{2}=&\lim\limits_{k\rightarrow\infty}\int_V a|\nabla u_k|^2+(\lambda_k h(x)+1)|u_k|^2\,d\mu\\=&a\lim\limits_{k\rightarrow\infty}\left(\int_V|\nabla u^+_k|^2\,d\mu+\int_V|\nabla u^-_k|^2\,d\mu-K_V(u_k)\right)\\&+\lim\limits_{k\rightarrow\infty}(\lambda_k h(x)+1)\left(\int_V|u^{+}_k|^2\,d\mu+\int_V|u^{-}_k|^2\,d\mu\right)\\=&a\left(\int_{\Omega\cup\partial\Omega}|\nabla u^+_0|^2\,d\mu+\int_{\Omega\cup\partial\Omega}|\nabla u^-_0|^2\,d\mu-K_\Omega(u_0)\right)\\&+\left(\int_\Omega|u^{+}_0|^2\,d\mu+\int_\Omega|u^{-}_0|^2\,d\mu\right)\\=&\int_{\Omega\cup\partial\Omega}a|\nabla u_0|^2\,d\mu+\int_\Omega|u_0|^2\,d\mu\\=&\|u_0\|^2_{\mathcal{H}_{0}^{1}(\Omega)}.
\end{aligned}
\end{equation}
Since $u_k\rightarrow u_0$ pointwise in $V$, we have that
$u_{k} \rightarrow u_{0}$ in $H^1(V).$

 In the following, we prove that $u_0$ is a least energy sign-changing solution of the equation (\ref{1.8}). By (\ref{0.8}), (\ref{0.7}), (\ref{0.9}), Lebesgue dominated theorem and $u_0|_{\Omega^{c}}=0$, one has that
$$
\begin{aligned}
o(1)=&(J_{\lambda_k}^{\prime}(u_k), u_k^{\pm})\\=&\|u_k^{\pm}\|_{\mathcal{H}_\lambda}^{2}+b\left\|\nabla u_k^{\pm}\right\|_{2}^{4}-\int_{V}\left|u_k^{\pm}\right|^{p} \log \left(u_k^{\pm}\right)^{2}\,d\mu-\frac{a}{2}K_{V}(u_k)+\frac{b}{2}K^2_V(u_k) 
\\&+b\left\|\nabla u_k^{+}\right\|_{2}^{2}\left\|\nabla u_k^{-}\right\|_{2}^{2}-\frac{b}{2}K_V(u_k)\left(\left\|\nabla u_k^{+}\right\|_{2}^{2}+\left\|\nabla u_k^{-}\right\|_{2}^{2}\right)
-bK_V(u_k)\|\nabla u_k^{\pm}\|^2_2\\=&\|u_0^{\pm}\|_{\mathcal{H}_{0}^{1}(\Omega)}^{2}+b\left\|\nabla u_0^{\pm}\right\|_{\ell^2(\Omega\cup\partial\Omega)}^{4}-\int_{\Omega}\left|u_0^{\pm}\right|^{p} \log \left(u_0^{\pm}\right)^{2}\,d\mu-\frac{a}{2}K_{\Omega}(u_0)+\frac{b}{2}K^2_\Omega(u_0) 
\\&+b\left\|\nabla u_0^{+}\right\|_{\ell^2(\Omega\cup\partial\Omega)}^{2}\left\|\nabla u_0^{-}\right\|_{\ell^2(\Omega\cup\partial\Omega)}^{2}-\frac{b}{2}K_\Omega(u_0)\left(\left\|\nabla u_0^{+}\right\|_{\ell^2(\Omega\cup\partial\Omega)}^{2}+\left\|\nabla u_0^{-}\right\|_{\ell^2(\Omega\cup\partial\Omega)}^{2}\right)
\\&-bK_\Omega(u_0)\|\nabla u_0^{\pm}\|^2_{\ell^2(\Omega\cup\partial\Omega)}+o(1)\\=&(J_\Omega(u_0),u_0^{\pm})+o(1).
\end{aligned}
$$
This means that $(J_\Omega(u_0),u_0^{\pm})=0.$ Hence $u_0\in\mathcal{M}_\Omega$.
Then it follows from (\ref{0.9}) that
$$
\begin{aligned}
\mathcal{M}_\Omega=&\lim\limits_{k\rightarrow\infty}J_{\lambda_k}\left(u_{k}\right) \\=&\lim\limits_{k\rightarrow\infty}\left[J_{\lambda_k}\left(u_{k}\right)-\frac{1}{p} (J_{\lambda_k}^{\prime}\left(u_{k}\right), u_{k})\right] \\
=&\lim\limits_{k\rightarrow\infty}\left[(\frac{1}{2}-\frac{1}{p})\|u_k\|_{\mathcal{H}_\lambda}^{2}+(\frac{1}{4}-\frac{1}{p})b\left\|\nabla u_k\right\|_{2}^{4}+\frac{2}{p^2}\int_V |u_k|^p\,d\mu\right] \\
=&\left[(\frac{1}{2}-\frac{1}{p})\|u_0\|_{\mathcal{H}_1^0(\Omega)}^{2}+(\frac{1}{4}-\frac{1}{p})b\left\|\nabla u_0\right\|_{\ell^2(\Omega\cup\partial\Omega)}^{4}+\frac{2}{p^2}\int_{\Omega} |u_0|^p\,d\mu\right]\\ =& J_{\Omega}(u_0)-\frac{1}{p}(J'_{\Omega}(u_0),u_0)\\=&J_\Omega(u_0).
\end{aligned}
$$
Hence $u_{0}$ is a least energy sign-changing solution of the equation (\ref{1.8}).\qed

\
\

{ \bf Declaration}
The author declares that there are no conflict of interest regarding the publication of this paper.

\
\

{ \bf Data availability}
No data was used for the research described in this paper.

\end{document}